\documentclass[11pt]{article}
\usepackage[utf8]{inputenc}
\usepackage[T1]{fontenc}
\usepackage[english]{babel}
\usepackage{amsmath}
\usepackage{amsfonts}
\usepackage{amssymb}
\usepackage{amsthm}
\usepackage{listings}
\usepackage{xcolor}
\usepackage{cite}
\usepackage{xypic}
\usepackage{pgf,tikz,pgfplots}
\usepackage{subcaption,ragged2e}
\pgfplotsset{compat=1.15}
\pgfplotsset{ticks = none}
\usepackage{mathrsfs}
\usetikzlibrary{arrows}
\usetikzlibrary{positioning}
\usepackage{verbatim}
\usepackage{mathtools}
\usepackage{multicol} 
\usepackage{hyperref}
\usepackage{comment}
\usepackage[left=1in, right=1in, bottom = 1in, top = 1in]{geometry}

\usepackage{changes}

\usepackage{pdflscape}
\usepackage{afterpage}
\usepackage{geometry}

\newcommand*{\N}{\mathbb{N}}
\newcommand*{\Z}{\mathbb{Z}}

\newcommand*{\R}{\mathbb{R}}

\renewcommand{\phi}{\varphi}
\renewcommand{\epsilon}{\varepsilon}
\renewcommand{\theta}{\vartheta}

\newtheorem{lem}{Lemma}[section]
\newtheorem{prop}{Proposition}[section]

\newtheorem{thm}{Theorem}[section]

\newcommand*{\desc}[3]{{\mathsf{des}}_{#1}(#2,#3)}
\newcommand*{\SG}{\mathsf{SG}}

\newcommand{\E}{\mathbb{E}}
\newcommand{\Prob}{\mathbb{P}}

\newcommand{\LE}{LE}

\newcommand{\nRole}[2]{{\eta}_{#1}^{(#2)}}
\newcommand{\noRole}[2]{{\overline{\eta}}_{#1}^{(#2)}}
\newcommand{\cornersymbol}{A}
\newcommand{\lcorner}[1]{\cornersymbol^{#1}_1}
\newcommand{\tcorner}[1]{\cornersymbol^{#1}_2}
\newcommand{\rcorner}[1]{\cornersymbol^{#1}_3}
\newcommand{\cutpointsymbol}{a}
\newcommand{\rcut}[1]{\cutpointsymbol^{#1}_1}
\newcommand{\bcut}[1]{\cutpointsymbol^{#1}_2}
\newcommand{\lcut}[1]{\cutpointsymbol^{#1}_3}
\newcommand{\probCorner}[2]{p_{#1}^{(#2)}}
\newcommand{\probVerticesSymbol}{p}
\newcommand{\probVertices}[2]{\probVerticesSymbol_{#1}^{(#2)}}
\newcommand{\onecomp}{\begin{tikzpicture}\fill[black] (0,0) -- (1/3,0)  -- (1/6,1.732/6);\end{tikzpicture}}
\newcommand{\twocompleft}{        \begin{tikzpicture}
            \fill[black] (2/9,0) -- (1/3,0) -- (1/6,1.732/6) -- (1/6-1/18,1.732/6-1.732/18);
            \fill[black] (0,0) -- (1/9,0) -- (1/18,1.732/18);
        \end{tikzpicture}}
\newcommand{\twocompup}{\begin{tikzpicture}
            \fill[black] (0,0) -- (1/3,0) -- (1/3-1/18,1.732/18) -- (1/18,1.732/18);
            \fill[black] (1/6,1.732/6) -- (1/6-1/18,1.732/6-1.732/18) -- (1/6+1/18,1.732/6-1.732/18);
        \end{tikzpicture}}
\newcommand{\twocompright}{\begin{tikzpicture}
                \fill[black] (0,0) -- (1/9,0) -- (2/9,1.732/9) -- (1/6,1.732/6);
                \fill[black] (1/3,0) -- (2/9,0) -- (1/18+2/9,1.732/18);
        \end{tikzpicture}}
\newcommand{\threecomp}{\begin{tikzpicture}
                \fill[black] (0,0) -- (1/9,0) -- (1/18,1.732/18);
                \fill[black] (1/6-1/18,1.732/6-1.732/18) -- (1/6,1.732/6) -- (1/6+1/18,1.732/6-1.732/18);
                \fill[black] (1/3,0) -- (2/9,0) -- (1/18+2/9,1.732/18);
        \end{tikzpicture}}

\newcommand{\treeSet}[1]{\mathcal{T}_{#1}}
\newcommand{\twoCompSetOne}[1]{\mathcal{S}_{#1}^1}
\newcommand{\twoCompSetTwo}[1]{\mathcal{S}_{#1}^2}
\newcommand{\twoCompSetThree}[1]{\mathcal{S}_{#1}^3}
\newcommand{\threeCompSet}[1]{\mathcal{R}_{#1}}

\title{Average height for Abelian sandpiles and the looping constant on Sierpi\'nski graphs}

\date{\today}
\author{Nico Heizmann, Robin Kaiser, Ecaterina Sava-Huss}
\begin{document}
\maketitle

\begin{abstract}
For the Abelian sandpile model on Sierpi\'nski graphs, we investigate several statistics such as average height, height probabilities and looping constant. In particular, we calculate the expected average height of a recurrent sandpile on the finite iterations of the Sierpi\'nski gasket and we also give an algorithmic approach for calculating the height probabilities of recurrent sandpiles under stationarity by using the connection between recurrent configurations of the Abelian sandpile Markov chain and uniform spanning trees. We also calculate the expected fraction of vertices of height $i$ for $i\in\{0,1,2,3\}$ of sandpiles under stationarity and relate the bulk average height to the looping constant on the Sierpi\'nski gasket.
\end{abstract}



\section{Introduction}

The Abelian sandpile model has its origins in \cite{btw}, where it was first introduced by Bak, Tang and Wiesenfeld as a model of self-organized criticality. Later on, Dhar \cite{dhar-burning} generalized the model to arbitrary finite graphs and called it \emph{the Abelian sandpile model}. He also investigated the algebraic structure of addition operators and described an one-to-one correspondence between the set of recurrent sandpiles and the set of spanning trees of the underlying graph; this correspondence is known under the name burning bijection or burning algorithm.
The model has since seen impressive progress on different state spaces, mostly on Euclidean lattices, where some of the hypotheses stemming from simulations of physicists have been meanwhile proven. See \cite{abelian-sandpile-desc} for an excellent survey on this matter. Other state spaces have not received the same amount of attention, and many questions still remain open. For instance, physicists have made predictions about the behavior of the Abelian sandpile model on the Sierpi\'nski graph more than 20 years; see \cite{physics1,physics2,physics3}. However, mathematically the Abelian sandpile model on state spaces of fractal nature is poorly understood. The limit shape of the Abelian sandpile model, when adding $n$ particles to the origin of the infinite 
Sierpi\'nski graph is investigated in \cite{chen-sandpile-limit-shape} and results concerning the group structure of the sandpile group have been considered in \cite{sandpile-group-gasket}, while the scaling limit of the identity element has been investigated in \cite{scaling-limit-gasket}.

One of the objects of interest in the study of Abelian sandpiles is the {\it sandpile Markov chain}, which in defined as follows. Consider a finite connected graph $G=(V\cup\{s\},E)$ with a distinguished vertex $s$ called the sink. Assign to each vertex $v\in V$ a natural number $\sigma(v)\in\mathbb{N}$  representing its mass, or the sandpile at $v$. We choose at every discrete time step a vertex $v\in V$ uniformly at random and add mass $1$ to it. If the resulting mass at $v$ is at least the number of neighbors of $v$, then we \emph{topple} $v$ by sending unit mass to each neighbor of $v$. Mass can leave the system through the sink, and the topplings will continue until all vertices are stable, that is, they have mass smaller than the number of neighbors. The sequence of consecutive topplings is called \emph{avalanche}. The processes of adding mass uniformly at random and toppling until having only stable vertices is a Markov chain on the finite set of stable configurations and the unique stationary measure for this Markov chain is the uniform distribution on the set of recurrent configurations. This set, together with the operation of pointwise addition followed by stabilization is a group, called the \emph{sandpile group} or the \emph{critical group}. There are various interesting questions in this context, for instance the size of an avalanche or the diameter distribution depending on the underlying graph.

\begin{figure}
    \centering
    \begin{subfigure}[t]{0.3\linewidth}
        \centering
        \resizebox{\linewidth}{!}{
        \begin{tikzpicture}[baseline=9ex]
        \node[shape=circle,draw=black,label=left:$\lcorner{0}$] (A) at (0,0) {};
        \node[shape=circle,draw=black,label=right:$\rcorner{0}$] (B) at (4,0) {};
        \node[shape=circle,draw=black,label=above:$\tcorner{0}$] (E) at (2,1.73*2) {};
        
        
        \path [-] (A) edge node[left] {} (B);
        \path [-] (A) edge node[left] {} (E);
        \path [-] (B) edge node[left] {} (E);
        \end{tikzpicture}
        }
    \end{subfigure}
    \begin{subfigure}[t]{0.3\linewidth}
        \centering
        \resizebox{\linewidth}{!}{
        \begin{tikzpicture}[baseline=9ex]
        \node[shape=circle,draw=black,label=left:$\lcorner{1}$] (A) at (0,0) {};
        \node[shape=circle,draw=black,label=below:$\bcut{1}$] (B) at (2,0) {};
        \node[shape=circle,draw=black,label=right:$\rcorner{1}$] (C) at (4,0) {};
        \node[shape=circle,draw=black,label=above left:$\lcut{1}$] (D) at (1,1.73) {};
        \node[shape=circle,draw=black,label=above right:$\rcut{1}$] (E) at (3,1.73) {};
        \node[shape=circle,draw=black,label=above:$\tcorner{1}$] (F) at (2,1.73*2) {} ;
        
        
        \path [-] (A) edge node[left] {} (B);
        \path [-] (A) edge node[left] {} (D);
        \path [-] (B) edge node[left] {} (D);
        \path [-] (B) edge node[left] {} (C);
        \path [-] (E) edge node[left] {} (B);
        \path [-] (E) edge node[left] {} (C);
        \path [-] (F) edge node[left] {} (E);
        \path [-] (F) edge node[left] {} (D);
        \path [-] (E) edge node[left] {} (D);
        \end{tikzpicture}}
    \end{subfigure}
    \begin{subfigure}[t]{0.3\linewidth}
        \centering
        \resizebox{\linewidth}{!}{
        \begin{tikzpicture}[baseline=9ex]
        \node[shape=circle,draw=black,label=left:$\lcorner{2}$] (A) at (0,0) {};
        \node[shape=circle,draw=black,label=below:$\bcut{2}$] (B) at (2,0) {};
        \node[shape=circle,draw=black,label=right:$\rcorner{2}$] (C) at (4,0) {};
        \node[shape=circle,draw=black,label=above left:$\lcut{2}$] (D) at (1,1.73) {};
        \node[shape=circle,draw=black,label=above right:$\rcut{2}$] (E) at (3,1.73) {};
        \node[shape=circle,draw=black,label=above:$\tcorner{2}$] (F) at (2,1.73*2) {} ;
        \node[shape=circle,draw=black] (G) at (1,0) {};
        \node[shape=circle,draw=black] (H) at (3,0) {};
        \node[shape=circle,draw=black] (I) at (1/2,1.73/2) {};
        \node[shape=circle,draw=black] (J) at (3/2,1.73/2) {};
        \node[shape=circle,draw=black] (K) at (5/2,1.73/2) {};
        \node[shape=circle,draw=black] (L) at (7/2,1.73/2) {};
        \node[shape=circle,draw=black] (M) at (3/2,1.73/2+1.73) {};
        \node[shape=circle,draw=black] (N) at (5/2,1.73/2+1.73) {};
        \node[shape=circle,draw=black] (O) at (2,1.73) {};
        
        
        \path [-] (A) edge node[left] {} (G);
        \path [-] (B) edge node[left] {} (G);
        \path [-] (A) edge node[left] {} (I);
        \path [-] (D) edge node[left] {} (I);
        \path [-] (B) edge node[left] {} (J);
        \path [-] (D) edge node[left] {} (J);
        \path [-] (I) edge node[left] {} (J);
        \path [-] (I) edge node[left] {} (G);
        \path [-] (J) edge node[left] {} (G);
        
        \path [-] (B) edge node[left] {} (H);
        \path [-] (C) edge node[left] {} (H);
        \path [-] (E) edge node[left] {} (K);
        \path [-] (E) edge node[left] {} (L);
        \path [-] (B) edge node[left] {} (K);
        \path [-] (C) edge node[left] {} (L);
        \path [-] (H) edge node[left] {} (K);
        \path [-] (H) edge node[left] {} (L);
        \path [-] (L) edge node[left] {} (K);
        
        \path [-] (F) edge node[left] {} (M);
        \path [-] (F) edge node[left] {} (N);
        \path [-] (E) edge node[left] {} (N);
        \path [-] (E) edge node[left] {} (O);
        \path [-] (D) edge node[left] {} (O);
        \path [-] (D) edge node[left] {} (M);
        \path [-] (M) edge node[left] {} (N);
        \path [-] (O) edge node[left] {} (N);
        \path [-] (O) edge node[left] {} (M);
        \end{tikzpicture}}
    \end{subfigure}
    \caption{The graphs $\SG_0$, $\SG_1$ and $\SG_2$.}
    \label{fig:first-3-it-SG}
\end{figure}

For a recurrent configuration chosen uniformly at random,  it is also of interest to understand the height distribution at some fixed vertex, that is, \emph{the height probabilities}. These height probabilities have been investigated on several state spaces so far. For instance on $\mathbb{Z}^2$, Priezzhev \cite{heights-on-zd} gave exact formulas for the height probabilities for the heights $1,2$ and $3$ in terms of rational polynomials in $1/\pi$ and multiple integrals; for a direct calculation of these integrals see \cite{exact_integral}. The ideas from \cite{heights-on-zd} have been extended in \cite{priezzhev-extension} to express the height probabilities in terms of a single integral, where also a simple formula for the height probabilities was conjectured. Using the connection between the average height of sandpiles and the looping constant as in \cite{loop-conn}, together with the computation of the looping constant in \cite{loop-const1} and in \cite{loop-const2}
confirms the conjecture from \cite{priezzhev-extension}, and the height probabilities on $\mathbb{Z}^2$ are given by:
$p_0=\frac{2}{\pi^2}-\frac{4}{\pi^3}$, $p_1=\frac{1}{4}-\frac{1}{2\pi}-\frac{3}{\pi^2}+\frac{12}{\pi^3}$, $p_2=\frac{3}{8}+\frac{1}{\pi}-\frac{12}{\pi^3}$, and finally $p_3=\frac{3}{8}-\frac{1}{2\pi}+\frac{1}{\pi^2}+\frac{4}{\pi^3}$.
The height probabilities for sandpiles on regular trees were calculated in \cite{bethe}.

The current work focuses on the height probabilities and expected height of recurrent sandpiles on the $n$-th level of the Sierpi\'nski graph, denoted by $\SG_n$, for  every $n\in\N$. See Section \ref{sec:Def-SG} for the precise definition of $\SG_n$ and Figure \ref{fig:first-3-it-SG} for an illustration of the first three levels. The methods used in \cite{heights-on-zd}  are not applicable to Sierpi\'nski graphs, but we use instead the connection between recurrent sandpiles and  spanning trees and forests of $\SG_n$, and this connection enables also to calculate the average number of vertices of a given height under stationarity; see Proposition \ref{prop:height-sandpiles} for details. Spanning trees and forests on $\SG_n$ exhibit a recursive structure as proven in \cite{ust-on-gasket}.
A similar recursive structure was also used to calculate the height probabilities on trees in \cite{bethe}. Note that, in contrast to regular trees \cite{bethe}, in our case the recursive structure is only apparent in the spanning trees. This is why our approach uses the burning bijection of \cite{dhar-burning}, which also builds on the concept of forbidden subconfigurations. Finally, we also investigate the connection between the Abelian sandpile and the looping constant similarly to \cite{looping-constant-of-zd}.  The main results are the following.

\begin{thm}\label{thm:loop-const}
For any $n\in\N$ and $v\in\SG_n$, let
\begin{align*}
    \zeta_n^{(v)}=\mathbb{E}\big[|\{\text{neighbours of $v$ visited by LERW on $\SG_n$ started from $v$}\}|\big],
\end{align*}
where LERW is the loop erased random walk on $\SG_n$ stopped after hitting either the bottom right vertex $A^n_3$ or the top corner vertex $A^n_2$. Further let
\begin{align*}
    \zeta_n=\frac{1}{|\SG_n|}\sum_{v\in \SG_n}\zeta_n^{(v)},
\end{align*}
and denote $\zeta:= lim_{n\to\infty}\zeta_n$. We then have 
\begin{align*}
  \zeta=\frac{7259}{5616}.
\end{align*}
\end{thm}
The proof of Theorem \ref{thm:loop-const} follows from Lemma \ref{lem:loop=desc} in Section \ref{sec:loop-cst} and the calculations from  Section \ref{sec:expHeight}. We also obtain the following result for the number of vertices of a given height in the Abelian sandpile model. Below $\sigma:\SG_n\to\mathbb{N}$ is a recurrent sandpile configuration on $\SG_n$ chosen uniformly at random from the set of all recurrent configurations, i.e. the set of recurrent states of the Abelian sandpile Markov chain on $\SG_n$, and $\mathbb{P}$ refers to the probability that $\sigma$ is chosen according to the stationary distribution on recurrent sandpiles of $\SG_n$.
\begin{thm}\label{prop:height-sandpiles}
On $\SG_n$, for any $n\in \mathbb{N}$ and sink vertex given by the top corner vertex $A^n_2$, for $i\in\{0,1,2,3\}$ let
\begin{align*}
    \overline{W}^i_n=\frac{1}{|\SG_n|}\sum_{v\in \SG_n}\mathbb{P}(\eta(v)=i),
\end{align*}
and
\begin{align*}
    \overline{W}_n=\frac{1}{|\SG_n|}\sum_{v\in \SG_n}\mathbb{E}\big[\sigma(v)\big].
\end{align*}
We then have
\begin{align*}
    \lim_{n\rightarrow\infty}\begin{pmatrix}
    \overline{W}_n^0\\\overline{W}_n^1\\\overline{W}_n^2\\\overline{W}_n^3\end{pmatrix}=\begin{pmatrix}
        10957/161856\\649680671/4222984896\\1448254439/4222984896\\1839170699/4222984896
    \end{pmatrix}\approx\begin{pmatrix}
        0.07\\0.15\\0.34\\0.44
    \end{pmatrix},
\end{align*}
and
\begin{align*}
    \lim_{n\rightarrow\infty}\overline{W}_n=\frac{24107}{11232}\approx 2.15.
\end{align*}
The same limits hold when we choose either two corners or all three corners as sink vertices.
\end{thm}
Theorem \ref{prop:height-sandpiles} follows from the calculations of Section \ref{sec:expHeight}. The paper is organized as follows.
In Section \ref{sec:prelim} we introduce the Abelian sandpile model and $\SG_n$, the Sierpi\'nski graphs of level $n\in \N$, and we describe the tools used throughout the paper. Of particular importance is the recursive structure of spanning forests of $\SG_n$. In
Section \ref{sec:height-prob} we describe an algorithmic approach to calculate the height probabilities of sandpiles under stationarity on $\SG_n$ based on the recursive decomposition of spanning forests described in Section \ref{sec:prelim}.
In Section \ref{sec:expHeight} we calculate the expected number of vertices of a given height as well as the expected bulk average height of a sandpile under stationarity.
In Section \ref{sec:loop-cst} we investigate the relation between the average looping constant and the expected average height of a sandpile under stationarity. Finally, in Appendix \ref{sec:appendixComp}, we collect the closed form expressions of several quantities used  through the paper.

\section{Preliminaries}\label{sec:prelim}
\subsection{Abelian sandpile model}

We refer the reader to \cite{abelian-sandpile-desc} for an extended survey on this topic. Let $G=(V\cup\{s\},E)$ be an undirected, connected and finite graph, where the vertex $s$ is called the sink. We denote by $\deg_G(v)$ the degree of vertex $v$ in G, that is the number of adjacent vertices of $v$, and when no confusion arises, we drop the subindex notation and write only $\deg(v)$.
A sandpile is a function $\sigma:V\rightarrow\Z$ and is to be interpreted as the number of particles sitting on each vertex. The sandpile $\sigma$ is called stable if $\sigma(v)<\deg(v)$, for all $v\in V$ and is called unstable at $v\in V$ if  $\sigma(v)\geq \deg(v)$, i.e. there are more particles at $v$ than connecting edges. We call $\sigma$ unstable if there exists $v\in V$ such that $\sigma$ is unstable at $v$.
Given a sandpile $\sigma$, we define the toppling at vertex $v\in V$ as
\begin{align*}
    T_v \sigma = \sigma - \Delta_G \delta_v,
\end{align*}
where $\delta_v:V\to\{0,1\}$ is the function taking the value $1$ at $v$ and $0$ everywhere else, and $\Delta_G$ is the graph Laplacian defined as
\vspace{-0.25cm}
\begin{align*}
    \Delta_G(x,y)=\begin{cases}
    \deg(x),&x=y\\
    -1,&x\sim_G y\\
    0,&\text{else}
    \end{cases},
\end{align*}
where $x\sim_G y$ means that $x$ and $y$ are adjacent in $G$. The toppling procedure distributes one particle from $v$ to each neighboring vertex. We say that the toppling at $v$ is legal if $\sigma$ is unstable at $v$.
Given an unstable sandpile $\sigma$, there always exists a sequence of vertices $v_1,...,v_n$ such that all the topplings at $v_1,...,v_n$ are legal and the configuration $T_{v_n}...T_{v_1}\sigma$ is stable. We then define the stabilization $\sigma^\circ$ of $\sigma$ as
\begin{align*}
    \sigma^\circ=T_{v_n}...T_{v_1}\sigma.
\end{align*}
Notice that the stabilization of $\sigma$ is unique, and thus the stabilization operation is well-defined. Using the notions of sandpiles and stabilization, we can now define a Markov chain that has as state space the set of stable sandpiles on $G$.

Let $X_1,X_2,...$ be i.i.d. random variables distributed uniformly on $V$ and let $\sigma_0$ be any stable sandpile configuration on $G$, the starting configuration. For any $n\in\N$ define
\begin{align*}
    \sigma_{n+1}=(\sigma_n+\delta_{X_n})^\circ.
\end{align*}
The sequence $(\sigma_n)_{n\in\N}$ is a Markov chain called the sandpile Markov chain and $\sigma_{n+1}$ is obtained from $\sigma_n$ by adding one chip uniformly at random on $V$ and stabilizing the new configuration. As shown in \cite{dhar-burning}, the set $\mathcal{R}_G$  of recurrent states  of the Markov chain $(\sigma_n)_{n\in \N}$ forms an Abelian group and the group operation $\oplus$ is given by: for $\sigma,\xi \in \mathcal{R}_G$
$$\sigma \oplus \xi := (\sigma + \xi)^\circ.$$ 
Restricted on the set $\mathcal{R}_G$, the sandpile Markov chain is
an irreducible random walk on a finite group, thus its  stationary distribution  is the uniform distribution on $\mathcal{R}_G$. Since the sandpile Markov chain ends up in the recurrent states after finitely many steps, it makes sense to start directly in stationarity, that is, to choose one sandpile $\sigma$ uniformly on $\mathcal{R}_G$ and to ask about the distribution of the number of chips we see at some vertex $v$, that is to investigate the \emph{height probabilities}  $\mathbb{P}(\sigma(v)=k)$ under the stationary distribution for all $k\in\{0,...,\deg(v)-1\}.$

\textbf{Multiple sinks.}
The underlying state spaces for the current paper are
the Sierpi\'nski graphs, where either a single vertex, two vertices or three vertices act as the sinks of the Abelian sandpile model. We describe here briefly what this means. For any  undirected, connected and finite graph $G=(V,E)$, let $S\subseteq V$ be a subset of vertices, the sinks of the Abelian sandpile model.
We define a modified version of $G$ denoted  by $G'=(V',E')$, with vertex set  $V'=V\backslash S \cup \{s\}$,
where $s$ is a new vertex not already in $V\backslash S$, and  the edge set $E'$ is
\begin{align*}
    E'=\big\{(x,y):x,y\in V\backslash S\text{ and }(x,y)\in E\big\}\cup\big\{(x,s):x\in V\text{ and }\exists y\in S:(x,y)\in E\big\}.
\end{align*}
That is, $G'=(V',E')$ is the graph where we identify all the vertices of the set $S$ to a single vertex $s$. When talking about the Abelian sandpile model on $G$ with sinks given by the vertices in $S$, we mean the Abelian sandpile model on this new graph $G'$.

\subsection{Burning algorithm}\label{sec:burning-algo}

We describe here the burning bijection due to Dhar \cite{dhar-burning}, which gives a bijective mapping from the set of recurrent sandpiles to the set of spanning trees of the underlying graph. This bijection plays a central role in this paper, as the calculations that follow are based on statistics of spanning trees and forests of $\SG_n$. In view of the burning bijection, we obtain then the height probabilities of the recurrent sandpiles. The following lemma lays the foundation of the burning bijection. 
\begin{lem}[\cite{dhar-burning}]\label{lem:burning-alg}
Let $\sigma$ be a recurrent sandpile on the graph $G=(V\cup\{s\},E)$. Let $x_1,...,x_n\in V$ be the vertices adjacent to the sink $s$, and for any $i\in\{1,...,n\}$ denote by $b_i$ the number of edges connecting $x_i$ to $s$. We  have
\begin{align*}
    (\sigma+\sum_{i=1}^n b_i\delta_{x_i})^\circ=\sigma,
\end{align*}
and during the stabilization every vertex topples exactly once.
\end{lem}
\begin{figure}
	    \begin{subfigure}[h]{0.5\linewidth}
		\centering
		\resizebox{\linewidth}{!}{
			\begin{tikzpicture}[baseline=9ex]
				\node[shape=circle,draw=black] (A) at (0,0) {};
				\node[below left= -0.5cm and -0.1cm of A, text width=0.8cm]{\footnotesize$a\!=\!1$ $b\!=\!0$} ;
				
				\node[shape=circle,draw=black] (B) at (2*2,2*0) {};
				\node[below= -0.01cm of B, text width=0.8cm]{\footnotesize$a\!=\!1$ $b\!=\!0$} ;
				
				\node[shape=circle,draw=black] (C) at (2*4,2*0) {};
				\node[below right = -0.5 and -0.01cm of C, text width=0.8cm]{\footnotesize$a\!=\!1$ $b\!=\!0$} ;
				
				\node[shape=circle,draw=black] (D) at (2*1,2*1.73) {};
				\node[above left = -0.5cm and -0.1cm of D, text width=0.8cm]{\footnotesize$a\!=\!0$ $b\!=\!0$} ;
				 
				\node[shape=circle,draw=black] (E) at (2*3,2*1.73) {};
				\node[above right = -0.5cm and -0.01cm of E, text width=0.8cm]{\footnotesize$a\!=\!1$ $b\!=\!0$} ;
				
				\node[shape=circle,draw=black] (F) at (2*2,2*1.73*2) {s};
				
				\node[shape=circle,draw=black] (G) at (2*1,0) {};
				\node[below= -0.01cm of G, text width=0.8cm]{\footnotesize$a\!=\!0$ $b\!=\!1$} ;
				\node[shape=circle,draw=black] (H) at (2*3,0) {};
				\node[below= -0.01cm of H, text width=0.8cm]{\footnotesize$a\!=\!1$ $b\!=\!0$} ;
				
				\node[shape=circle,draw=black] (I) at (2*1/2,2*1.73/2) {};
				\node[above left = -0.5cm and -0.1cm of I, text width=0.8cm]{\footnotesize$a\!=\!0$ $b\!=\!0$} ;
				\node[shape=circle,draw=black] (J) at (2*3/2,2*1.73/2) {};
				\node[above right = -0.5cm and -0.01cm of J, text width=0.8cm]{\footnotesize$a\!=\!0$ $b\!=\!0$} ;
				\node[shape=circle,draw=black] (K) at (2*5/2,2*1.73/2) {};
				\node[above left = -0.5cm and -0.1cm of K, text width=0.8cm]{\footnotesize$a\!=\!1$ $b\!=\!0$} ;
				\node[shape=circle,draw=black] (L) at (2*7/2,2*1.73/2) {};
				\node[above right = -0.5cm and -0.01cm of L, text width=0.8cm]{\footnotesize$a\!=\!0$ $b\!=\!0$} ;
				
				\node[shape=circle,draw=black] (M) at (2*3/2,2*1.73/2+2*1.73) {};
				\node[above left = -0.5cm and -0.1cm of M, text width=0.8cm]{\footnotesize$a\!=\!0$ $b\!=\!0$} ;
				\node[shape=circle,draw=black] (N) at (2*5/2,2*1.73/2+2*1.73) {};
				\node[above right = -0.5cm and -0.01cm of N, text width=0.8cm]{\footnotesize$a\!=\!0$ $b\!=\!0$} ;
				\node[shape=circle,draw=black] (O) at (2*2,2*1.73) {};
				\node[below= -0.01cm of O, text width=0.8cm]{\footnotesize$a\!=\!0$ $b\!=\!0$} ;
				
				\path [-,gray] (A) edge node[left] {} (G);
				\path [-,gray] (B) edge node[left] {} (G);
				\path [-,gray] (A) edge node[left] {} (I);
				\path [-,gray] (D) edge node[left] {} (I);
				\path [-,gray] (B) edge node[left] {} (J);
				\path [-,gray] (D) edge node[left] {} (J);
				\path [-,gray] (I) edge node[left] {} (J);
				\path [-,gray] (I) edge node[left] {} (G);
				\path [-,gray] (J) edge node[left] {} (G);
				
				\path [-,gray] (B) edge node[left] {} (H);
				\path [-,gray] (C) edge node[left] {} (H);
				\path [-,gray] (E) edge node[left] {} (K);
				\path [-,gray] (E) edge node[left] {} (L);
				\path [-,gray] (B) edge node[left] {} (K);
				\path [-,gray] (C) edge node[left] {} (L);
				\path [-,gray] (H) edge node[left] {} (K);
				\path [-,gray] (H) edge node[left] {} (L);
				\path [-,gray] (L) edge node[left] {} (K);
				
				\path [-,gray] (F) edge node[left] {} (M);
				\path [-,gray] (F) edge node[left] {} (N);
				\path [-,gray] (E) edge node[left] {} (N);
				\path [-,gray] (E) edge node[left] {} (O);
				\path [-,gray] (D) edge node[left] {} (O);
				\path [-,gray] (D) edge node[left] {} (M);
				\path [-,gray] (M) edge node[left] {} (N);
				\path [-,gray] (O) edge node[left] {} (N);
				\path [-,gray] (O) edge node[left] {} (M);
				 
				\draw[ultra thick] (F) -- (N) -- (O) -- (E) -- (L) -- (K) -- (H) -- (C);
				\draw[ultra thick] (F) -- (M) -- (D) -- (I) -- (G);
				\draw[ultra thick] (A) -- (G) -- (B);
				\draw[ultra thick] (J) -- (D);
		\end{tikzpicture}}
	\end{subfigure}
	\begin{subfigure}[h]{0.5\linewidth}
		\centering
		\resizebox{\linewidth}{!}{
			\begin{tikzpicture}[baseline=9ex]
				\node[shape=circle,draw=black] (A) at (0,0) {};
				\node[below left= -0.2cm and -0.1cm of A, text width=0.3cm]{0} ;
				\node[below left= -0.5cm and -0.1cm of A, text width=0.8cm, text opacity=0]{{\footnotesize$a\!=\!0$ $b\!=\!0$}} ;
				
				\node[shape=circle,draw=black] (B) at (2*2,2*0) {};
				\node[below= -0.01cm of B]{2} ;
				\node[below= -0.01cm of B, text width=0.8cm, text opacity=0]{{\footnotesize$a\!=\!1$ $b\!=\!0$}} ;
				
				\node[shape=circle,draw=black] (C) at (2*4,2*0) {};
				\node[below right = -0.2cm and -0.01cm of C, text width=0.3cm]{0} ;
				\node[below right = -0.5 and -0.01cm of C, text width=0.8cm, text opacity=0]{{\footnotesize$a\!=\!1$ $b\!=\!0$}} ;
								
				\node[shape=circle,draw=black] (D) at (2*1,2*1.73) {};
				\node[above left = -0.2cm and -0.01cm of D]{3} ;
				
				\node[shape=circle,draw=black] (E) at (2*3,2*1.73) {};
				\node[above right = -0.2cm and -0.01cm of E]{2} ;
				
				\node[shape=circle,draw=black] (F) at (2*2,2*1.73*2) {s};
				
				\node[shape=circle,draw=black] (G) at (2*1,0) {};
				\node[below= -0.01cm of G]{2} ;
				\node[shape=circle,draw=black] (H) at (2*3,0) {};
				\node[below= -0.01cm of H]{2} ;
				
				\node[shape=circle,draw=black] (I) at (2*1/2,2*1.73/2) {};
				\node[above left = -0.2cm and -0.01cm of I]{3} ;
				\node[shape=circle,draw=black] (J) at (2*3/2,2*1.73/2) {};
				\node[above right = -0.2cm and -0.01cm of J]{3} ;
				\node[shape=circle,draw=black] (K) at (2*5/2,2*1.73/2) {};
				\node[above left = -0.2cm and -0.01cm of K]{2} ;
				\node[shape=circle,draw=black] (L) at (2*7/2,2*1.73/2) {};
				\node[above right = -0.2cm and -0.01cm of L]{3} ;
				
				\node[shape=circle,draw=black] (M) at (2*3/2,2*1.73/2+2*1.73) {};
				\node[above left = -0.2cm and -0.01cm of M]{3} ;
				\node[shape=circle,draw=black] (N) at (2*5/2,2*1.73/2+2*1.73) {};
				\node[above right = -0.2cm and -0.01cm of N]{3} ;
				\node[shape=circle,draw=black] (O) at (2*2,2*1.73) {};
				\node[below= -0.01cm of O]{3} ;
				
				\path [-] (A) edge node[left] {} (G);
				\path [-] (B) edge node[left] {} (G);
				\path [-] (A) edge node[left] {} (I);
				\path [-] (D) edge node[left] {} (I);
				\path [-] (B) edge node[left] {} (J);
				\path [-] (D) edge node[left] {} (J);
				\path [-] (I) edge node[left] {} (J);
				\path [-] (I) edge node[left] {} (G);
				\path [-] (J) edge node[left] {} (G);
				
				\path [-] (B) edge node[left] {} (H);
				\path [-] (C) edge node[left] {} (H);
				\path [-] (E) edge node[left] {} (K);
				\path [-] (E) edge node[left] {} (L);
				\path [-] (B) edge node[left] {} (K);
				\path [-] (C) edge node[left] {} (L);
				\path [-] (H) edge node[left] {} (K);
				\path [-] (H) edge node[left] {} (L);
				\path [-] (L) edge node[left] {} (K);
				
				\path [-] (F) edge node[left] {} (M);
				\path [-] (F) edge node[left] {} (N);
				\path [-] (E) edge node[left] {} (N);
				\path [-] (E) edge node[left] {} (O);
				\path [-] (D) edge node[left] {} (O);
				\path [-] (D) edge node[left] {} (M);
				\path [-] (M) edge node[left] {} (N);
				\path [-] (O) edge node[left] {} (N);
				\path [-] (O) edge node[left] {} (M);

		\end{tikzpicture}}
	\end{subfigure}
	\caption{The burning bijection: the spanning tree with statistics $a_T, b_T$ (left) and its corresponding recurrent sandpile (right). The total ordering of $E_v$ is determined by the number of clockwise rotations by $\pi/3$ needed to align the edge with $(1,0)$, with fewer rotations indicating a lower position in the ordering.}
	\label{fig:burning-bij}
\end{figure}

For $v\in V$, denote by $E_v$ the set of edges incident to $v$ and fix a total ordering $<_v$ of all the edges in $E_v$. In the original burning algorithm spanning trees are constructed by the order of topplings during the stabilization in Lemma \ref{lem:burning-alg}. The following version of the burning bijection from \cite{looping-constant-of-zd} defines the inverse. Given a spanning tree $T$ of $G$, for every $v$ there is a unique path connecting $v$ to the sink $s$. Denote by $e_T(v)$ the first edge and  by $l_T(v)$ the number of edges on this path and let
\begin{align*}
    &a_T(v)=\#\{(v,y)\in E_v  : \ l_T(y)<l_T(v)-1\},\\
    &b_T(v)=\#\{(v,y)\in E_v  : \ l_T(y)=l_T(v)-1\text{ and }(v,y)<_v e_T(v)\}.
\end{align*}
Then the sandpile defined by
\begin{align*}
    \sigma_T(v)=\deg_G(v)-1-a_T(v)-b_T(v)
\end{align*}
is recurrent and the mapping $T\mapsto \sigma_T$ is bijective.  See Figure \ref{fig:burning-bij} for an illustration of $a_T$, $b_T$ and the corresponding sandpile $\sigma_T$, for one particular spanning tree of the Sierpi\'nski graph of level 3.
Given a spanning tree $T$ of $G$, we say that $v$ is descendant of $w$ in $T$ if the unique path from $v$ to $s$ in  $T$ contains $w$, shortly $v<_T w$. We call $w$ an ascendant of $v$ in $T$, and we denote the number of neighboring descendants by
 $$\desc{}{T}{v} = \#\{(v,y) \in E_v : \ y<_T v\}.$$
The following lemma, whose proof can be found in \cite[Lemma 4]{looping-constant-of-zd}, gives a way to calculate height probabilities of recurrent sandpiles under the stationary distribution of the sandpile Markov chain from the number of neighbours that are descendants in the uniform spanning tree of $G$. The uniform spanning tree of $G$, denoted by $\mathrm{UST}$, is a random variable distributed uniformly on the set of spanning trees of $G$. That is, if $\tau_G$ is the number of spanning trees of $G$, then for any spanning tree $T$ on $G$ we have
\begin{align*}
    \mathbb{P}(\mathrm{UST}=T)=\frac{1}{\tau_G}.
\end{align*}

\begin{lem}[\cite{heights-on-zd}]\label{lem:desc-conn-to-height}
For any vertex $v\neq s$ and any $0\leq j \leq k \leq \deg_G(v)-1$ we have
\begin{align*}
    \mathbb{P}(\sigma_{\mathrm{UST}}(v)=k \ | \  \desc{}{\mathrm{UST}}{v} = j )=\frac{1}{\deg_G(v)-j},
\end{align*}
where $\mathrm{UST}$ denotes the uniform spanning tree on $G$.
\end{lem}
For another proof see \cite[Lemma 4]{looping-constant-of-zd}.

\subsection{Sierpi\'nski graphs}\label{sec:Def-SG}

We introduce below Sierpi\'nski graphs and we
describe the iterative construction of their spanning trees as in \cite{ust-on-gasket}.

\paragraph{Construction of the finite iterations of the Sierpi\'nski graph.}

For a graph $G=(V,E)$ that can be embedded in $\R^2$ and $x\in\R^2$, we define $x+G$ to be the graph with vertex set
$x+V=\{x+y:y\in V\}$,  and edge set
$x+E=\{(x+a,x+b):\ (a,b)\in E\}$.
Let $\SG_0$ be the graph with vertex set $V_0$ and edge set $E_0$ given by
\begin{align*}
    &V_0=\Big\{\big(0,0\big),\big(1,0\big),\frac{1}{2}\big(1,\sqrt{3}\big)\Big\},\\
    &E_0=\Big\{\Big\{\big(0,0\big),\big(1,0\big)\Big\},\Big\{\big(1,0\big),\frac{1}{2}\big(1,\sqrt{3}\big)\Big\},\Big\{\big(0,0\big),\frac{1}{2}\big(1,\sqrt{3}\big)\Big\}\Big\}.
\end{align*}
For $n\geq 1$, the level $n$ Sierpi\'nski graph $\SG_n=(V_n,E_n)$ is defined inductively by
\begin{align*}
    \SG_n=\SG_{n-1}\cup\Big(\big(2^n,0\big)+\SG_{n-1}\Big)\cup\Big(\big(2^{n-1},2^{n-1}\sqrt{3}\big)+\SG_{n-1}\Big).
\end{align*}

That is, we take three copies of $\SG_{n-1}$, shift one to the right and one diagonally, and then take the union of these copies to obtain $\SG_n$. See Figure \ref{fig:first-3-it-SG} for a graphical representation of the first three steps of this inductive process. The (infinite) Sierpi\'nski graph is then defined as $\SG = \bigcup_{n\in \N} \SG_n$.\\
\\
Special vertices of interest are the three corners of $\SG_n$, which will be denoted by $\lcorner{n},\rcorner{n}$ and $\tcorner{n}$, as well as the three cut points opposing the corner vertices, denoted by $\rcut{n},\bcut{n}$ and $\lcut{n}$. See again Figure \ref{fig:first-3-it-SG} for an illustration of these special vertices in the first three iterations of the Sierpi\'nski graphs.
An immediate consequence of the iterative nature of $\SG_n$ is that
$$|V_n|=\frac{3}{2}\big(3^n+1\big)\quad \text{and}\quad |E_n|=3^{n+1}.$$

\paragraph{Spanning trees and forests of $\SG_n$.}
Lemma \ref{lem:desc-conn-to-height} will be used below to calculate the heights of recurrent sandpiles by looking at the number of neighbours that are descendants in spanning trees, or spanning forests respectively, on $\SG_n$. Analyzing the later statistic is possible due to a recursive description of the spanning trees and forests on $\SG_n$ through the spanning trees and forests on $\SG_{n-1}$ as derived in \cite{ust-on-gasket}. 

\begin{figure}
    \centering
    \begin{tabular}{c|c|c|c|c}
         $\treeSet{n}$ & $ \twoCompSetOne{n}$ & $ \twoCompSetTwo{n}$ & $ \twoCompSetThree{n}$ & $\threeCompSet{n}$ \\[5pt] \hline &&&& \\[-8pt]
        \begin{tikzpicture}
            \fill[black] (0,0) -- (1,0)  -- (1/2,1.732/2);
        \end{tikzpicture}
        &
        \begin{tikzpicture}
            \fill[black] (2/3,0) -- (1,0) -- (1/2,1.732/2) -- (1/2-1/6,1.732/2-1.732/6);
            \fill[black] (0,0) -- (1/3,0) -- (1/6,1.732/6);
        \end{tikzpicture}
        &
        \begin{tikzpicture}
            \fill[black] (0,0) -- (1,0) -- (1-1/6,1.732/6) -- (1/6,1.732/6);
            \fill[black] (1/2,1.732/2) -- (1/2-1/6,1.732/2-1.732/6) -- (1/2+1/6,1.732/2-1.732/6);
        \end{tikzpicture}
        &
        \begin{tikzpicture}
                \fill[black] (0,0) -- (1/3,0) -- (2/3,1.732/3) -- (1/2,1.732/2);
                \fill[black] (1,0) -- (2/3,0) -- (1/6+2/3,1.732/6);
        \end{tikzpicture}
        &
        \begin{tikzpicture}
                \fill[black] (0,0) -- (1/3,0) -- (1/6,1.732/6);
                \fill[black] (1/2-1/6,1.732/2-1.732/6) -- (1/2,1.732/2) -- (1/2+1/6,1.732/2-1.732/6);
                \fill[black] (1,0) -- (2/3,0) -- (1/6+2/3,1.732/6);
        \end{tikzpicture}
    \end{tabular}
    \caption{Pictograms for one, two, and three component forests}
    \label{fig:forestIllust}
\end{figure}

Denote by $\mathcal{T}_n$ the set of spanning trees of $\SG_n$, and  let $\mathcal{S}_n^i$ be the set of spanning forests of $\SG_n$ consisting of two connected components, where $A^n_i$ lies in its own connected component and the other two corner vertices lie in the other connected component. Finally, let $\mathcal{R}_n$ be the set of spanning forests of $\SG_n$ consisting of three connected components, where every corner lies in its own connected component. We will use the pictographic representation of these sets as described in Figure \ref{fig:forestIllust}.
The key property of these trees and forests is that they decompose into three trees or forests on the three copies of $\SG_{n-1}$ in $\SG_n$. For example, any element of $\mathcal{T}_n$ decomposes into two trees of $\SG_{n-1}$ and a suitable choice of a two component forest as can be seen in Figure \ref{fig:recursive_spanningTree}. The same can be done for elements of $\mathcal{S}_n^i$ and $\mathcal{R}_n$, albeit there are more possibilities than in the case of $\mathcal{T}_n$. In Figure \ref{fig:recursive_twoComponent} we list all the ways to decompose elements of $\mathcal{S}_n^2$. By suitable rotations, we also obtain the decomposition of elements of $\mathcal{S}_n^i$ for $i=1$ and $i=3$. In Figure \ref{fig:recursive_threeComponent}, we illustrate up to rotations and reflections all possible decompositions of elements of $\mathcal{R}_n$. For details and proofs of these decompositions we refer the reader to \cite{ust-on-gasket}. Using this recursive construction of spanning trees and forests of $\SG_n$, we thus also obtain a closed form expression for the number of trees and forests of $\SG_n$. 
Let
\begin{align*}
    \tau_n:=|\mathcal{T}_n|,&&\sigma_n:=|\mathcal{S}_n^1|=|\mathcal{S}_n^2|=|\mathcal{S}_n^3|,&&\rho_n:=|\mathcal{R}_n|.
\end{align*}
Then, by \cite[Lemma 4.1]{ust-on-gasket}, the recursions are given by
\begin{align*}
    \tau_{n+1} = 6 \tau_n^2 \sigma_n, \quad \sigma_{n+1} = 7 \tau_n \sigma_n^2+\tau_n^2\rho_n , \quad \rho_{n+1} = 14 \sigma_n^2 + 12 \tau_n\sigma_n \rho_n,
\end{align*}
with solutions are 
\begin{align}
    &\tau_n=3\Bigg(\frac{5}{3}\Bigg)^{-n/2}540^{\frac{3^n-1}{4}},\label{eq:tau}\\
    &\sigma_n=\Bigg(\frac{5}{3}\Bigg)^{n/2}540^{\frac{3^n-1}{4}},\label{eq:sigma}\\
    &\rho_n=\Bigg(\frac{5}{3}\Bigg)^{3n/2}540^{\frac{3^n-1}{4}}.\label{eq:rho}
\end{align}
\begin{figure}
    \centering
    \begin{align*}
    \begin{tikzpicture}
        \fill[black] (0,0) -- (1/3,0) -- (2/3,1.732/3) -- (1/2,1.732/2);
        \fill[black] (1,0) -- (2/3,0) -- (1/6+2/3,1.732/6);
        \fill[black] (2,0) -- (1,0) -- (1/2+1,1.732/2);
        \fill[black] (1,1.732) -- (1/2+1,1.732/2) -- (1/2,1.732/2);
    \end{tikzpicture}
    &&
    \begin{tikzpicture}
        \fill[black] (0,0) -- (1,0) -- (1/2,1.732/2);
        \fill[black] (1,0) -- (1+1/3,0) -- (1+1/6,1.732/6);
        \fill[black] (1+2/3,0) -- (2,0) -- (1/2+1,1.732/2) -- (1/2+1-1/6,1.732/2-1.732/6);
        \fill[black] (1,1.732) -- (1/2+1,1.732/2) -- (1/2,1.732/2);
    \end{tikzpicture}
    &&
    \begin{tikzpicture}
        \fill[black] (0,0) -- (1,0) -- (1/2,1.732/2);
        \fill[black] (1/2+1,1.732/2) -- (3/2-1/6,1.732/2-1.732/6) -- (3/2+1/6,1.732/2-1.732/6);
        \fill[black] (1,0) -- (2,0) -- (2-1/6,1.732/6) -- (1+1/6,1.732/6);
        \fill[black] (1,1.732) -- (1/2+1,1.732/2) -- (1/2,1.732/2);
    \end{tikzpicture}
    \\
    \begin{tikzpicture}
        \fill[black] (0,0) -- (1,0) -- (1-1/6,1.732/6) -- (1/6,1.732/6);
        \fill[black] (1/2,1.732/2) -- (1/2-1/6,1.732/2-1.732/6) -- (1/2+1/6,1.732/2-1.732/6);
        \fill[black] (2,0) -- (1,0) -- (1/2+1,1.732/2);
        \fill[black] (1,1.732) -- (1/2+1,1.732/2) -- (1/2,1.732/2);
    \end{tikzpicture}
    &&
    \begin{tikzpicture}
        \fill[black] (0,0) -- (1,0) -- (1/2,1.732/2);
        \fill[black] (1/2,1.732/2) -- (1/2+1/3,1.732/2) -- (1+1/6,1.732-1.732/6) -- (1,1.732);
        \fill[black] (2,0) -- (1,0) -- (1/2+1,1.732/2);
        \fill[black] (3/2,1.732/2) -- (3/2-1/6,1.732/2+1.732/6) -- (3/2-1/3,1.732/2);
    \end{tikzpicture}
    &&
    \begin{tikzpicture}
        \fill[black] (0,0) -- (1,0) -- (1/2,1.732/2);
        \fill[black] (3/2,1.732/2) -- (3/2-1/3,1.732/2) -- (1-1/6,1.732-1.732/6) -- (1,1.732);
        \fill[black] (2,0) -- (1,0) -- (1/2+1,1.732/2);
        \fill[black] (1/2,1.732/2) -- (1/2+1/6,1.732/2+1.732/6) -- (1/2+1/3,1.732/2);
    \end{tikzpicture}
\end{align*}
    \caption{All possible configurations for $\treeSet{n}$}
    \label{fig:recursive_spanningTree}
\end{figure}

\begin{figure}
    \centering\begin{align*}
    \begin{tikzpicture}
        \fill[black] (0,0) -- (1,0) -- (1/2,1.732/2);
        \fill[black] (3/2,1.732/2) -- (3/2-1/3,1.732/2) -- (1-1/6,1.732-1.732/6) -- (1,1.732);
        \fill[black] (1/2,1.732/2) -- (1/2+1/6,1.732/2+1.732/6) -- (1/2+1/3,1.732/2);
        \fill[black] (1/2+1,1.732/2) -- (3/2-1/6,1.732/2-1.732/6) -- (3/2+1/6,1.732/2-1.732/6);
        \fill[black] (1,0) -- (2,0) -- (2-1/6,1.732/6) -- (1+1/6,1.732/6);
    \end{tikzpicture}
    &&
    \begin{tikzpicture}
        \fill[black] (0,0) -- (1,0) -- (1/2,1.732/2);
        \fill[black] (1,0) -- (1+1/3,0) -- (1+1/6,1.732/6);
        \fill[black] (1+2/3,0) -- (2,0) -- (1/2+1,1.732/2) -- (1/2+1-1/6,1.732/2-1.732/6);
        \fill[black] (1,1.732) -- (1+1/6,1.732-1.732/6) -- (1-1/6,1.732-1.732/6);
        \fill[black] (1/2,1.732/2) -- (1/2+1/6,1.732/2+1.732/6) -- (3/2-1/6,1.732/2+1.732/6) -- (3/2,1.732/2);
    \end{tikzpicture}
    &&
    \begin{tikzpicture}
        \fill[black] (0,0) -- (1,0) -- (1/2,1.732/2);
        \fill[black] (1/2+1,1.732/2) -- (3/2-1/6,1.732/2-1.732/6) -- (3/2+1/6,1.732/2-1.732/6);
        \fill[black] (1,0) -- (2,0) -- (2-1/6,1.732/6) -- (1+1/6,1.732/6);
        \fill[black] (1,1.732) -- (1+1/6,1.732-1.732/6) -- (1-1/6,1.732-1.732/6);
        \fill[black] (1/2,1.732/2) -- (1/2+1/6,1.732/2+1.732/6) -- (3/2-1/6,1.732/2+1.732/6) -- (3/2,1.732/2);
    \end{tikzpicture}\\
    \begin{tikzpicture}
        \fill[black] (2,0) -- (1,0) -- (1/2+1,1.732/2);
        \fill[black] (1/2,1.732/2) -- (1/2+1/3,1.732/2) -- (1+1/6,1.732-1.732/6) -- (1,1.732);
        \fill[black] (3/2,1.732/2) -- (3/2-1/6,1.732/2+1.732/6) -- (3/2-1/3,1.732/2);
        \fill[black] (0,0) -- (1,0) -- (1-1/6,1.732/6) -- (1/6,1.732/6);
        \fill[black] (1/2,1.732/2) -- (1/2-1/6,1.732/2-1.732/6) -- (1/2+1/6,1.732/2-1.732/6);
    \end{tikzpicture}
    &&
    \begin{tikzpicture}
        \fill[black] (2,0) -- (1,0) -- (1/2+1,1.732/2);
        \fill[black] (0,0) -- (1/3,0) -- (2/3,1.732/3) -- (1/2,1.732/2);
        \fill[black] (1,0) -- (2/3,0) -- (1/6+2/3,1.732/6);
        \fill[black] (1,1.732) -- (1+1/6,1.732-1.732/6) -- (1-1/6,1.732-1.732/6);
        \fill[black] (1/2,1.732/2) -- (1/2+1/6,1.732/2+1.732/6) -- (3/2-1/6,1.732/2+1.732/6) -- (3/2,1.732/2);
    \end{tikzpicture}
    &&
    \begin{tikzpicture}
        \fill[black] (2,0) -- (1,0) -- (1/2+1,1.732/2);
        \fill[black] (0,0) -- (1,0) -- (1-1/6,1.732/6) -- (1/6,1.732/6);
        \fill[black] (1/2,1.732/2) -- (1/2-1/6,1.732/2-1.732/6) -- (1/2+1/6,1.732/2-1.732/6);
        \fill[black] (1,1.732) -- (1+1/6,1.732-1.732/6) -- (1-1/6,1.732-1.732/6);
        \fill[black] (1/2,1.732/2) -- (1/2+1/6,1.732/2+1.732/6) -- (3/2-1/6,1.732/2+1.732/6) -- (3/2,1.732/2);
    \end{tikzpicture}\\
    \begin{tikzpicture}
        \fill[black] (0,0) -- (1,0) -- (1-1/6,1.732/6) -- (1/6,1.732/6);
        \fill[black] (1/2,1.732/2) -- (1/2-1/6,1.732/2-1.732/6) -- (1/2+1/6,1.732/2-1.732/6);
        \fill[black] (1/2+1,1.732/2) -- (3/2-1/6,1.732/2-1.732/6) -- (3/2+1/6,1.732/2-1.732/6);
        \fill[black] (1,0) -- (2,0) -- (2-1/6,1.732/6) -- (1+1/6,1.732/6);
        \fill[black] (1,1.732) -- (1/2+1,1.732/2) -- (1/2,1.732/2);
    \end{tikzpicture}
    &&
    \begin{tikzpicture}
        \fill[black] (2,0) -- (1,0) -- (1/2+1,1.732/2);
        \fill[black] (0,0) -- (1,0) -- (1/2,1.732/2);
        \fill[black] (1/2,1.732/2) -- (1/2+1/3,1.732/2) -- (1/2+1/6,1.732/2+1.732/6);
        \fill[black] (1,1.732) -- (1-1/6,1.732-1.732/6) -- (1+1/6,1.732-1.732/6);
        \fill[black] (3/2,1.732/2) -- (3/2-1/3,1.732/2) -- (3/2-1/6,1.732/2+1.732/6);
    \end{tikzpicture}
\end{align*}
    \caption{All possible configurations for $\twoCompSetTwo{n}$}
    \label{fig:recursive_twoComponent}
\end{figure}
\vspace{-0.25cm}

In order to highlight on which iteration of the Sierpi\'nski graph we are currently working on, we denote the number of neighbours that are descendants in a tree respectively in forest $t$ of $\SG_n$ - previously defined as $\mathsf{des}(t,\cdot)$ - by $\desc{n}{t}{\cdot}$, that is, we put the level of the graph in the subscript.

\paragraph{Choice of the sink.}
Depending on how we choose the sink vertices in $\SG_n$, we obtain a bijection between the recurrent sandpiles of the graph and the sets of spanning forests or trees previously described. The effect of the choice of the sink vertices (and of their number) will be visible after applying the burning bijection. In particular, a single sink vertex corresponds to the root of the spanning tree after applying the burning bijection; if we identify two different vertices as the sink, then we obtain a spanning forest with two connected components, whose roots will be given by the two sink vertices. Choosing more sink vertices will result in more connected components in the spanning forest obtained after applying the burning bijection. Throughout this paper, we will consider the following choices for the sink in the Sierpi\'nski gasket.
The first possibility is to let the sink be any one of
the three corner vertices, and then the recurrent sandpiles  will correspond to
spanning trees i.e. elements of $\mathcal{T}_n$ under the burning bijection.
Secondly, we can choose two of the corners and identify them as the sink of $\SG_n$. In this case, we get that the recurrent sandpiles are in bijection with elements of $\mathcal{S}_n^i\cup\mathcal{S}_n^j$ for some $i,j\in\{1,2,3\}$, where $i$ and $j$ depend on the choice of the two sink vertices. For example, when letting $\tcorner{n}$ and $\rcorner{n}$ be the two sinks, we get a bijection with $\mathcal{S}_n^2\cup\mathcal{S}_n^3$.
Finally, we could also declare all three corners as our sinks, in which case we end up with a bijection between the recurrent sandpile and the set $\mathcal{R}_n$. For the remainder of the paper, we will make thorough use of the iterative construction of spanning trees and forests on the Sierpinski gasket graphs as shown in Figures \ref{fig:recursive_spanningTree},  \ref{fig:recursive_twoComponent}, and \ref{fig:recursive_threeComponent}. This will give us different formulas for the height probabilities in the three different cases of choosing the sink.

\begin{figure}
    \centering
    \begin{align*}
    \begin{tikzpicture}
        \fill[black] (1,0) -- (1+1/3,0) -- (1+1/6,1.732/6);
        \fill[black] (1+2/3,0) -- (2,0) -- (1/2+1,1.732/2) -- (1/2+1-1/6,1.732/2-1.732/6);
        \fill[black] (0,0) -- (1,0) -- (1/2,1.732/2);
        \fill[black] (1/2,1.732/2) -- (1/2+1/3,1.732/2) -- (1/2+1/6,1.732/2+1.732/6);
        \fill[black] (1,1.732) -- (1-1/6,1.732-1.732/6) -- (1+1/6,1.732-1.732/6);
        \fill[black] (3/2,1.732/2) -- (3/2-1/3,1.732/2) -- (3/2-1/6,1.732/2+1.732/6);
    \end{tikzpicture}\times 6
    &&
    \begin{tikzpicture}
        \fill[black] (1,0) -- (1+1/3,0) -- (3/2+1/6,1.732/2-1.732/6) -- (3/2,1.732/2);
        \fill[black] (2,0) -- (2-1/3,0) -- (2-1/6,1.732/6);
        \fill[black] (0,0) -- (1,0) -- (1/2,1.732/2);
        \fill[black] (1/2,1.732/2) -- (1/2+1/3,1.732/2) -- (1/2+1/6,1.732/2+1.732/6);
        \fill[black] (1,1.732) -- (1-1/6,1.732-1.732/6) -- (1+1/6,1.732-1.732/6);
        \fill[black] (3/2,1.732/2) -- (3/2-1/3,1.732/2) -- (3/2-1/6,1.732/2+1.732/6);
    \end{tikzpicture}\times 6
    &&
    \begin{tikzpicture}
        \fill[black] (1,0) -- (1+1/3,0) -- (3/2+1/6,1.732/2-1.732/6) -- (3/2,1.732/2);
        \fill[black] (2,0) -- (2-1/3,0) -- (2-1/6,1.732/6);
        \fill[black] (0,0) -- (1,0) -- (1-1/6,1.732/6) -- (1/6,1.732/6);
        \fill[black] (1/2,1.732/2) -- (1/2-1/6,1.732/2-1.732/6) -- (1/2+1/6,1.732/2-1.732/6);
        \fill[black] (1/2,1.732/2) -- (1/2+1/3,1.732/2) -- (1+1/6,1.732-1.732/6) -- (1,1.732);
        \fill[black] (3/2,1.732/2) -- (3/2-1/6,1.732/2+1.732/6) -- (3/2-1/3,1.732/2);
    \end{tikzpicture}\times 6
    \\
    \begin{tikzpicture}
        \fill[black] (1,0) -- (1+1/3,0) -- (3/2+1/6,1.732/2-1.732/6) -- (3/2,1.732/2);
        \fill[black] (2,0) -- (2-1/3,0) -- (2-1/6,1.732/6);
        \fill[black] (0,0) -- (1,0) -- (1-1/6,1.732/6) -- (1/6,1.732/6);
        \fill[black] (1/2,1.732/2) -- (1/2-1/6,1.732/2-1.732/6) -- (1/2+1/6,1.732/2-1.732/6);
        \fill[black] (1,1.732) -- (1+1/6,1.732-1.732/6) -- (1-1/6,1.732-1.732/6);
        \fill[black] (1/2,1.732/2) -- (1/2+1/6,1.732/2+1.732/6) -- (3/2-1/6,1.732/2+1.732/6) -- (3/2,1.732/2);
    \end{tikzpicture}\times 6
    &&
    \begin{tikzpicture}
        \fill[black] (1,0) -- (1+1/3,0) -- (1+1/6,1.732/6);
        \fill[black] (1+2/3,0) -- (2,0) -- (1/2+1,1.732/2) -- (1/2+1-1/6,1.732/2-1.732/6);
        \fill[black] (0,0) -- (1,0) -- (1-1/6,1.732/6) -- (1/6,1.732/6);
        \fill[black] (1/2,1.732/2) -- (1/2-1/6,1.732/2-1.732/6) -- (1/2+1/6,1.732/2-1.732/6);
        \fill[black] (1/2,1.732/2) -- (1/2+1/3,1.732/2) -- (1+1/6,1.732-1.732/6) -- (1,1.732);
        \fill[black] (3/2,1.732/2) -- (3/2-1/6,1.732/2+1.732/6) -- (3/2-1/3,1.732/2);
    \end{tikzpicture}\times 2
\end{align*}
    \caption{Up to rotation and reflection, all possible configurations for $\threeCompSet{n}$ }
    \label{fig:recursive_threeComponent}
\end{figure}
\section{Height probabilities}\label{sec:height-prob}

We calculate here the height probabilities for corner vertices and cut points, and we give an algorithmic approach to calculate the height probabilities for any other vertex in $\SG_n$. In order to do so, we  use the connection to the number of neighbours that are descendants in an uniformly chosen spanning tree (resp. forest) as shown in Lemma \ref{lem:desc-conn-to-height}.
More precisely,  for any $v\in \SG_n$ we calculate the probability that exactly $k$ neighbors of $v$ are descendants in the spanning forest for the three different sink configurations$$\Prob(\desc{n}{\mathrm{UST}}{v}=k), \quad 0\leq k < \deg(v).$$
We do so by looking first at the probabilities of the roots, corners and cut points in each iteration of the Sierpi\'nski graph $\SG_n$. We then calculate the height probabilities of the remaining vertices by combining the height probabilities of the previous iterations in the three subtriangles of $\SG_n$ for every $n\in\N$. This is possible, since cutpoints act as roots or corner vertices in the subtriangles and the neighboring descendants of the remaining vertices stay the descendants in the subtriangles.
We denote by $\Prob$ the uniform measure on the set $\mathcal{Q}_n := \mathcal{T}_n\cup \mathcal{S}_n^1 \cup \mathcal{S}_n^2\cup \mathcal{S}_n^3\cup \mathcal{R}_n$. Conditioning on the number of components results again in a uniform measure, i.e. $\Prob( \cdot \ | \ t\in \mathcal{T}_n)$ is the uniform measure on $\mathcal{T}_n$.

\subsection{Probabilities at corner points}
We first calculate the probabilities for the various numbers of neighboring descendants at non-root corner points, in both the tree and the 2-component forest settings. This corresponds to the height probabilities in the ASM with the single root as the sink and the two roots as a multiple sink respectively. Note that by symmetry, the probabilities at the corner point $\lcorner{n}$ for forests in $\mathcal{S}_n^2$ and $\mathcal{S}_n^3$ are the same. Furthermore the symmetry also yields that the probabilities at the corner points $\lcorner{n}$ and $\rcorner{n}$ for trees $\mathcal{T}_n$ are equal. 
We denote the probabilities of corner points having $k$ neighboring descendants by
\begin{align*}
    \probCorner{1}{n}(k) = \Prob\big(\desc{n}{t}{\lcorner{n}} = k \ | \ t \in \mathcal{T}_n\big), \quad \probCorner{2}{n}(k) = \Prob\big(\desc{n}{t}{\lcorner{n}} = k \ | \ t \in \mathcal{S}_n^2\big), \quad k=0,1.
\end{align*}
By going through all the cases shown in Figure \ref{fig:recursive_spanningTree} and Figure \ref{fig:recursive_twoComponent} we get \begin{align*}
    \probCorner{1}{n}(k) &= 4 \frac{\tau_{n-1}^2\sigma_{n-1}}{\tau_n} \probCorner{1}{n-1}(k) + 2 \frac{\tau_{n-1}^2\sigma_{n-1}}{\tau_n} \probCorner{2}{n-1}(k), \\
    \probCorner{2}{n}(k) &= 3 \frac{\tau_{n-1}\sigma_{n-1}^2}{\sigma_n} \probCorner{1}{n-1}(k) + 4 \frac{\tau_{n-1}\sigma_{n-1}^2}{\sigma_n} \probCorner{2}{n-1}(k) + \frac{\tau_{n-1}^2\rho_{n-1}}{\sigma_n} \probCorner{1}{n-1}(k),
\end{align*}
which together with equations (\ref{eq:tau})-(\ref{eq:rho}) implies
\begin{align*}
    \begin{pmatrix}
    \probCorner{1}{n}(k)\\\probCorner{2}{n}(k)
    \end{pmatrix}
    =
    \begin{pmatrix}
    2/3 && 1/3\\
    3/5 && 2/5\\
    \end{pmatrix}^{n}
    \begin{pmatrix}
    \probCorner{1}{0}(k)\\
    \probCorner{1}{0}(k)
    \end{pmatrix}.
\end{align*}
The powers of the matrix in the equation above can be calculated by the diagonalization method and are given by
\begin{align*}
        \begin{pmatrix}
    2/3 && 1/3\\
    3/5 && 2/5\\
    \end{pmatrix}^{n}
    =
    \frac{15^{-n}}{14}\begin{pmatrix}
        3^{n+2}\cdot5^n+5 && -5(1-15^n)\\
        -9(1-15^n) && 3^n\cdot 5^{n+1} + 9
    \end{pmatrix}.
\end{align*}
We can then finally calculate the probabilities for $\SG_0$ by going through all the cases of spanning trees and 2-component forests, in order to obtain the initial values
\begin{align*}
    \probCorner{1}{0}(0)&=2/3,&\probCorner{1}{0}(1)&=1/3\\
    \probCorner{2}{0}(0)&=1,&\probCorner{2}{0}(1)&=0,
\end{align*}
and for $n\geq 1$
\begin{align*}
    \probCorner{1}{n}(0) = \frac{11}{14} - \frac{5}{42} 15^{-n}, \quad p_1^{(n)}(1) = \frac{3}{14} + \frac{5}{42} 15^{-n}, \\
    \probCorner{2}{n}(0) = \frac{11}{14} + \frac{3}{14} 15^{-n}, \quad p_2^{(n)}(1) = \frac{3}{14} - \frac{3}{14} 15^{-n}.
\end{align*}

\subsection{Neighbours in the same component}\label{subsec:neighborsSameComp}

Next, we want to calculate  the probabilities of the number of descendants for the root vertices in 2-component and 3-component spanning forests of $\SG_n$. Notice that for roots, a descendant vertex is exactly a vertex that lies in the same component as the root, which is the basis of this section's title. 
Although the roots of the forests act as the sinks in the ASM and therefore do not posses any form of height, the calculations made here are crucial for our arguments because the cutpoints may act as roots in the subforests for the decomposition of $\SG_{n}$ into three copies of $\SG_{n-1}$. We need to distinguish the cases for $\twoCompSetTwo{n}$ and $\twoCompSetThree{n}$ for 2-component forests, since they appear a different number of times in the construction of 2- and 3-component forests.
For this purpose denote for $k=0,1,2$
\begin{align*}
    \nRole{2}{n}(k) &= \Prob\big( \desc{n}{t}{\rcorner{n}} = k \ | \ t \in \twoCompSetTwo{n}\big),\\
    \noRole{2}{n}(k) &= \Prob\big( \desc{n}{t}{\tcorner{n}} = k \ | \ t \in \twoCompSetTwo{n}\big),\\
    \nRole{3}{n}(k) &= \Prob\big( \desc{n}{t}{\tcorner{n}} = k \ | \ t \in \threeCompSet{n}\big).
\end{align*}


Going through all the cases from Figure \ref{fig:recursive_twoComponent} and Figure \ref{fig:recursive_threeComponent} we obtain the following linear recursion:
\begin{align*}
    \begin{pmatrix}
        \nRole{2}{n}(k)\\ \noRole{2}{n}(k) \\ \nRole{3}{n}(k)
    \end{pmatrix}
    =
    \begin{pmatrix}
        {12/30} && 0 && 0\\
        6/30 && 12/30 && 9/30\\
        14/50 && 12/50 && 12/50
    \end{pmatrix}
    \begin{pmatrix}
        \nRole{2}{n-1}(k)\\ \noRole{2}{n-1}(k) \\ \nRole{3}{n-1}(k)
    \end{pmatrix}
    +
    \begin{pmatrix}
        {18/30}\\ 3/30 \\ 12/50
    \end{pmatrix}\delta_2(k).
\end{align*}
We can again calculate the powers of the matrix by  an eigenvalue decomposition in order to obtain
\begin{align*}
    \begin{pmatrix}
        112\cdot 5^{-n} && 0 && 0\\
        5^{-2n}(-133\cdot 5^n +26\cdot 3^{n+1}\cdot5^n+55) && 8\cdot5^{-2n}(3^{n+2}\cdot5^n+5) && -12\cdot5^{1-2n}(1-15^n)\\
        2\cdot5^{-2n}(7\cdot5^n+26\cdot15^n-33) && -48\cdot5^{-2n}(1-15^n) && 8\cdot5^{-2n}(3^n\cdot5^{n+1}+9)
    \end{pmatrix}
\end{align*}
as the $n$-th power of the matrix in the linear recursion. In order to solve this recursion, we use again the probabilities for the 0-th iteration, given by
\begin{align*}
    \nRole{2}{0}(0)=0,&&\nRole{2}{0}(1)=1,&&\nRole{2}{0}(2)=0,\\
    \noRole{2}{0}(0)=1,&&\noRole{2}{0}(1)=0,&&\noRole{2}{0}(2)=0,\\
    \nRole{3}{0}(0)=1,&&\nRole{3}{0}(1)=0,&&\nRole{3}{0}(2)=0.
\end{align*}
For the exact solutions of  $\nRole{n}{2},\noRole{n}{2},\nRole{n}{3}$ see Table \ref{tab:rootProbs}.
\begin{table}
    \centering
    \makebox[\textwidth][c]{
    \begin{tabular}[b]{c||c|c|c}
       & $k=0$ & $k=1$ & $k=2$ \\ \hline
        $\nRole{2}{n}(k)$ & 0 & $\left(\frac{2}{5}\right)^n$ & $1-\left(\frac{2}{5}\right)^n$ \\
        $\noRole{2}{n}(k)$ & $\frac{33}{28} \left(\frac{3}{5}\right)^n-\frac{5}{28} \left(\frac{1}{25}\right)^n$ & $\frac{39}{28} \left(\frac{3}{5}\right)^n - \frac{29}{18} \left(\frac{2}{5}\right)^n + \frac{55}{252} \left(\frac{1}{25}\right)^n$ & $1-\frac{18}{7} \left(\frac{3}{5}\right)^n+\frac{29}{18} \left(\frac{2}{5}\right)^n - \frac{5}{126}\left(\frac{1}{25}\right)^n$ \\
       $\nRole{3}{n}(k)$&  $\frac{11}{14} \left(\frac{3}{5}\right)^n-\frac{3}{14} \left(\frac{1}{25}\right)^n$ & $\frac{39}{42} \left(\frac{3}{5}\right)^n - \frac{28}{42} \left(\frac{2}{5}\right)^n + \frac{11}{42} \left(\frac{1}{25}\right)^n$ & $1-\frac{12}{7} \left(\frac{3}{5}\right)^n+\frac{2}{3} \left(\frac{2}{5}\right)^n - \frac{1}{21}\left(\frac{1}{25}\right)^n$ 
    \end{tabular}
    }
    \caption{Exact values of the probabilities for roots $\nRole{n}{2}, \noRole{n}{2}, \nRole{n}{3}$.}
    \label{tab:rootProbs}
\end{table}

\subsection{Probabilities at cut points}\label{subsection:cut-points}

Denote the cut points of the iteration $n$ by $\rcut{n},\bcut{n}$ and $\lcut{n}$ as in Figure \ref{fig:first-3-it-SG}. We can compute their respective probabilities of having $k$ descendants in a spanning tree, 2-component or 3-component forest respectively, by using the probabilities for cut points and number of neighbours in the same component calculated previously. We will briefly explain the procedure on the basis of the spanning trees. The general case works the same by going through all the combinations depicted in Figure \ref{fig:recursive_spanningTree}, \ref{fig:recursive_twoComponent} and \ref{fig:recursive_threeComponent}.
We consider the lower cut points $\bcut{n}$  and the following two cases:
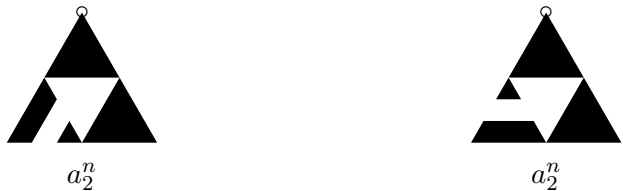
\begin{align*}
    \begin{tikzpicture}
        \fill[black] (0,0) -- (1/3,0) -- (2/3,1.732/3) -- (1/2,1.732/2);
        \fill[black] (1,0) -- (2/3,0) -- (1/6+2/3,1.732/6);
        \fill[black] (2,0) -- (1,0) -- (1/2+1,1.732/2);
        \fill[black] (1,1.732) -- (1/2+1,1.732/2) -- (1/2,1.732/2);
        \node[label=below:$\bcut{n}$] (B) at (1,0) {};
        \node[] (B) at (1,1.732) {$\circ$};
    \end{tikzpicture}
    &&
\begin{tikzpicture}
        \fill[black] (0,0) -- (1,0) -- (1-1/6,1.732/6) -- (1/6,1.732/6);
        \fill[black] (1/2,1.732/2) -- (1/2-1/6,1.732/2-1.732/6) -- (1/2+1/6,1.732/2-1.732/6);
        \fill[black] (2,0) -- (1,0) -- (1/2+1,1.732/2);
        \fill[black] (1,1.732) -- (1/2+1,1.732/2) -- (1/2,1.732/2);
        \node[label=below:$\bcut{n}$] (B) at (1,0) {};
        \node[] (B) at (1,1.732) {$\circ$};
    \end{tikzpicture}
\end{align*}
Consider first the neighbours of the right sub triangle. The path from them to the root $\tcorner{n}$  cannot go through the left sub triangle, hence it must go through the top corner of the smaller copy on the right. But this means that, if the neighbours are descendants of the cut point, then the same is true for the smaller spanning tree in the right copy and viceversa.
Now the unique path from the neighbours of the left triangle can either go directly to the top or through the right triangle. If it goes through the right triangle, then it must cross the cut point, in which case the neighbours are descendants of the cut point. If it does not go through the cut point, then the neighbours must lie in the other connected component of the two component spanning forest in the smaller left triangle. So we see that the number of descendants is simply the number of descendants in the right triangle combined with the neighbours of the cut point in the left triangle that lie in the same connected component of the spanning forest in the left triangle.
Now the same observations are true for the cases:
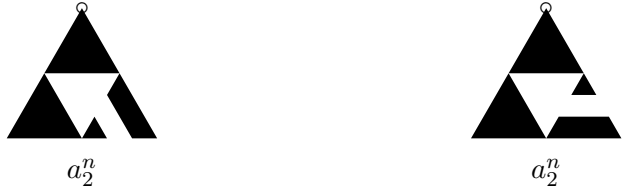
\begin{align*}
    \begin{tikzpicture}
        \fill[black] (0,0) -- (1,0) -- (1/2,1.732/2);
        \fill[black] (1,0) -- (1+1/3,0) -- (1+1/6,1.732/6);
        \fill[black] (1+2/3,0) -- (2,0) -- (1/2+1,1.732/2) -- (1/2+1-1/6,1.732/2-1.732/6);
        \fill[black] (1,1.732) -- (1/2+1,1.732/2) -- (1/2,1.732/2);
        \node[label=below:$\bcut{n}$] (B) at (1,0) {};
        \node[] (B) at (1,1.732) {$\circ$};
    \end{tikzpicture}
    &&
    \begin{tikzpicture}
        \fill[black] (0,0) -- (1,0) -- (1/2,1.732/2);
        \fill[black] (1/2+1,1.732/2) -- (3/2-1/6,1.732/2-1.732/6) -- (3/2+1/6,1.732/2-1.732/6);
        \fill[black] (1,0) -- (2,0) -- (2-1/6,1.732/6) -- (1+1/6,1.732/6);
        \fill[black] (1,1.732) -- (1/2+1,1.732/2) -- (1/2,1.732/2);
        \node[label=below:$\bcut{n}$] (B) at (1,0) {};
        \node[] (B) at (1,1.732) {$\circ$};
    \end{tikzpicture}
\end{align*}
Finally let us consider the last two cases:
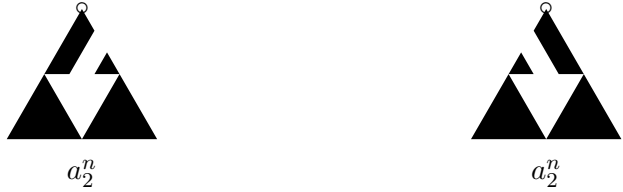
\begin{align*}
    \begin{tikzpicture}
        \fill[black] (0,0) -- (1,0) -- (1/2,1.732/2);
        \fill[black] (1/2,1.732/2) -- (1/2+1/3,1.732/2) -- (1+1/6,1.732-1.732/6) -- (1,1.732);
        \fill[black] (2,0) -- (1,0) -- (1/2+1,1.732/2);
        \fill[black] (3/2,1.732/2) -- (3/2-1/6,1.732/2+1.732/6) -- (3/2-1/3,1.732/2);
        \node[label=below:$\bcut{n}$] (B) at (1,0) {};
        \node[] (B) at (1,1.732) {$\circ$};
    \end{tikzpicture}
    &&
    \begin{tikzpicture}
        \fill[black] (0,0) -- (1,0) -- (1/2,1.732/2);
        \fill[black] (3/2,1.732/2) -- (3/2-1/3,1.732/2) -- (1-1/6,1.732-1.732/6) -- (1,1.732);
        \fill[black] (2,0) -- (1,0) -- (1/2+1,1.732/2);
        \fill[black] (1/2,1.732/2) -- (1/2+1/6,1.732/2+1.732/6) -- (1/2+1/3,1.732/2);
        \node[label=below:$\bcut{n}$] (B) at (1,0) {};
        \node[] (B) at (1,1.732) {$\circ$};
    \end{tikzpicture}
\end{align*}
For the case on the left, the unique path from all the points in the small right triangle to the top corner must go through the left triangle, hence both neighbours of the cut point in the right triangle are descendants of the cut point. For the neighbours in the left triangle we again observe that the number of descendants is simply the number of descendants when we consider the spanning tree on the left. Thus we obtain that the number of descendants is two plus the number of descendants in the left. For the second case on the right we make the same observations after switching the roles of the left and right triangle. We thus obtain the following equation for the probability that $\bcut{n}$ has $k$ neighbours as descendants in a spanning tree:
\begin{align*}
    \mathbb{P}(\desc{n}{T}{\cutpointsymbol_{2}^{n}} =&k \ | \ T\in\mathcal{T}_n)=\\
    &\frac{2}{3}\Bigg(\sum_{i=0}^k p_1^{(n-1)}(i)\cdot \frac{\eta_2^{(n-1)}(k-i)+\overline{\eta}_2^{(n-1)}(k-i)}{2}\Bigg)+\frac{1}{3}p_1^{(n-1)}(k-2).
\end{align*}
Notice that all the probabilities on the right-hand side above have been calculated in the previous subsections, hence we can calculate the probabilities for $\bcut{n}$. By symmetry, $\lcut{n}$ and $\rcut{n}$ have the same probabilities, hence it suffices to calculate the probabilities for $\lcut{n}$. Again going through all the cases for spanning trees we obtain
\begin{align*}
    \mathbb{P}(\desc{n}{T}{\cutpointsymbol_{3}^{n}} =k|T\in\mathcal{T}_n)=   
    &\frac{1}{6}\Bigg(\sum_{i=0}^k \Big[2p_1^{(n-1)}(i)\noRole{2}{n-1}(k-i)+p_1^{(n-1)}(i)\nRole{2}{n-1}(k-i)\Big]\\&+2p_1^{(n-1)}(k-2)+p_2^{(n-1)}(k-2)\Bigg).
\end{align*}
Going through all the cases in Figure \ref{fig:recursive_twoComponent}, we can also calculate the probabilities for the cut points in a spanning forest of type $\mathcal{S}_n^2$. For $\bcut{n}$ we have
\begin{align*}
    \mathbb{P}\big(\desc{n}{T}{\bcut{n}} = &k \ | \  T \in \mathcal{S}_n^2 \big)=&\\
    &\frac{1}{10}\Bigg(\sum_{i=0}^k\Big[2p_1^{(n-1)}(i)\noRole{2}{n-1}(k-i)+2p_1^{(n-1)}(i)\nRole{2}{n-1}(k-i)\\&+p_2^{(n-1)}(i)\nRole{2}{n-1}(k-i)\Big]+2p_2^{(n-1)}(k-2)\Bigg)+\frac{3}{10}p_1^{(n-1)}(k-2),
\end{align*}
while for $\rcut{n}$ we obtain
\begin{align*}
    \mathbb{P}\big(\desc{n}{T}{\rcut{n}} = &k \ | \  T \in \mathcal{S}_{n}^2 \big)=\\
    &\frac{1}{10}\sum_{i=0}^k\Big[2p_1^{(n-1)}(i)\noRole{2}{n-1}(k-i)+2p_1^{(n-1)}(i)\nRole{2}{n-1}(k-i)\\&+p_2^{(n-1)}(i)\nRole{2}{n-1}(k-i)+2p_2^{(n-1)}(i)\noRole{2}{n-1}(k-i)\Big]\\&+\frac{3}{10}\sum_{i=0}^k p_1^{(n-1)}(i)\nRole{3}{n-1}(k-i),
\end{align*}
and finally for $\lcut{n}$
\begin{align*}
    \mathbb{P}\big(\desc{n}{T}{\lcut{n}} = &k \ | \ T \in \mathcal{S}_{n}^2 \big) =\\
    & \frac{1}{10} \Bigg( \sum_{k=0}^k \Big[ 2 p_1^{(n-1)}(i) \noRole{2}{n-1}(k-i) + p_1^{(n-1)}(i) \nRole{2}{n-1}(k-i) \\ & + p_2^{(n-1)} (i)\nRole{2}{n-1}(k-i) +2 p_2^{(n-1)}(i) \noRole{2}{n-1}(k-i) \Big] + p_2^{(n-1)}(k-2) \Bigg) \\ & + \frac{3}{10} \sum_{i=0}^k p_1^{(n-1)}(i) \nRole{3}{n-1}(k-i).
\end{align*}
For the cut points in a three component forest, we have by symmetry that they all have the same probabilities, hence it suffices to do the calculations for $\bcut{n}$. Again by the same approach as above and going through all the cases in Figure \ref{fig:recursive_threeComponent} we obtain
\begin{align*}
    \mathbb{P}\big(\desc{n}{T}{\bcut{n}} = &k \ | \  T \in \mathcal{R}_n \big)=\\
    &\frac{3}{50}\sum_{i=0}^k\Bigg[4p_1^{(n-1)}(i)\eta_3^{(n-1)}(k-i)+2p_1^{(n-1)}(i)\noRole{2}{n-1}(k-i)\\
    &+4p_2^{(n-1)}(i)\eta_3^{(n-1)}(k-i)+2p_1^{(n-1)}(i)\nRole{2}{n-1}(k-i)\Bigg]\\
    &+\frac{1}{50}\sum_{i=0}^k\Bigg[8p_2^{(n-1)}(i)\noRole{2}{n-1}(k-i)+6p_2^{(n-1)}(i)\nRole{2}{n-1}(k-i)\Bigg].
\end{align*}
The probabilities for cut points can be calculated and  we collect the exact values in Appendix \ref{sec:appendixComp} but omit the simple proof of induction.

\subsection{Probabilities for arbitrary vertices}

Finally, we can calculate the probabilities for all vertices of the level $n$ Sierpi\'nski graph $\SG_n$. This can be done inductively, and we describe here our approach.
For all vertices in $B_n = \{ \lcorner{n}, \tcorner{n}, \rcorner{n}, \rcut{n}, \bcut{n}, \lcut{n} \}$ the probabilities can be calculated as elaborated in the previous sections. 
All other vertices are contained in exactly one of the three copies of $\SG_{n-1}$ in $\SG_n$, denoted by $\SG_{n-1}^L,\SG_{n-1}^U, \SG_{n-1}^R$ as the left, upper, and right sub triangle of $\SG_n$, respectively.
Let us assume that $v$ is in $\SG_{n-1}^L$ within $\SG_n$. Then for any given $t \in \mathcal{Q}_n$, the number of neighbours that are descendants of $v$ within $t$ is the same as the number of neighbours that are descendants of $v$ within the subforest of $t$ in $\SG_{n-1}^L$. Hence, we can once again obtain the probabilities for $v$ by counting the number of appearances of trees, 2-component and 3-component spanning forests within the corresponding copy of $\SG_{n-1}$ within $\SG_n$.
We denote for any $v\in \SG_n$
\begin{align*}
    \probVertices{1}{n}(v) &= \Prob(\desc{n}{t}{v} \ | \ t \in \treeSet{n}), \\
    \probVertices{2}{n}(v) &= \Prob(\desc{n}{t}{v} \ | \ t \in \twoCompSetTwo{n}), \\
    \probVertices{3}{n}(v) &= \Prob(\desc{n}{t}{v} \ | \ t \in \threeCompSet{n}),
\end{align*}
and let  $r_n:\SG_{n}\rightarrow \SG_{n}$ be a clockwise rotation by $120^\circ$, $m_1$ be a reflection along an axis such that $\twoCompSetOne{n-1}$ stays invariant and choose $m_2,m_3$ as reflection accordingly for $\twoCompSetTwo{n-1}$ and $\twoCompSetThree{n-1}$.
By symmetry we have 
\begin{align*}
    \Prob\left(\desc{n}{t}{v} \ | \ t \in \twoCompSetOne{n}\right) = \probVertices{2}{n}(r_n(v)), \quad
    \Prob\left(\desc{n}{t}{v} \ | \ t \in \twoCompSetThree{n}\right) = \probVertices{2}{n}(r_n^{-1}(v)).
\end{align*}
Write $\probVertices{i,d}{n}$ for the restriction of $\probVertices{i}{n}$ to $\SG_{n-1}^d \backslash B_n$ and $\phi_d$ for the restriction of the natural mapping from $\SG_{n-1}^d$ to $\SG_{n-1}$ where $i=1,2,3$ and $d=L,U,R$. Then according to Figure \ref{fig:recursive_spanningTree} we get for the probabilities of the left sub triangle
\begin{align*}
    &\probVertices{1,L}{n} = \frac{1}{6}\left(3\probVertices{1}{n-1}+\probVertices{1}{n-1}\circ r^{-1}_{n-1}+\probVertices{2}{n-1} + \probVertices{2}{n-1}\circ m_1 \right) \circ \phi_L.
\end{align*}
In the same manner we can calculate the probabilities in the lower right and upper triangles
\begin{align*}
    &p_{1,R}^{(n)}=\frac{1}{6} \left(3\probVertices{1}{n-1}+\probVertices{1}{n-1}\circ r_{n-1}+\probVertices{2}{n-1} \circ r_{n-1}+ \probVertices{2}{n-1}\circ m_2 \right)\circ \phi_R,\\
    &p_{1,U}^{(n)}=\frac{1}{6}\left(4\probVertices{1}{n-1}+\probVertices{2}{n-1} \circ r_{n-1}+\probVertices{2}{n-1} \circ m_1 \right) \circ \phi_U.
\end{align*}
For 2-component forests we obtain
\begin{align*}
    &\probVertices{2,L}{n}=\frac{1}{10}\left(\probVertices{1}{n-1}+5\probVertices{1}{n-1}\circ r_{n-1}^{-1}+3\probVertices{2}{n-1}+\probVertices{2}{n-1}\circ m_1\right)\circ \phi_L,\\ 
    &\probVertices{2,R}{n}=\frac{1}{10}\left(6 \probVertices{1}{n-1}\circ r^{-1}_{n-1}+3\probVertices{2}{n-1}+\probVertices{2}{n-1} \circ m_3 \right) \circ \phi_R,\\
    &\probVertices{2,U}{n}=\frac{1}{10}\left(\probVertices{1}{n-2}+3\probVertices{2}{n-1}+\probVertices{2}{n-1} \circ r_{n-1}+\probVertices{2}{n-1} \circ m_1 + \probVertices{2}{n-1} \circ m_2+3\probVertices{3}{n-1} \right)\circ \phi_U.
\end{align*}
Finally, for 3-component forests we have
\begin{align*}
	\probVertices{3,L}{n}=&\frac{1}{50} \big(12 \probVertices{3}{n-1}+12 \probVertices{1}{n-1} \circ r+7 \probVertices{2}{n-1} \circ m_2\\
	&+7 \probVertices{2}{n-1} \circ r_{n-1}^{-1} + 6 \probVertices{2}{n-1} \circ m_3 +6\probVertices{2}{n} \circ r_{n-1} \big)\circ \phi_L,\\
    p_{3,R}^{(n)}=&p_{3,L}^{(n)}\circ r_n \circ \phi_R,\\
    p_{3,U}^{(n)}=&p_{3,L}^{(n)}\circ r^{-1}_n \circ \phi_U.
\end{align*}
This now describes a recursive algorithm with which we can calculate the height probabilities up to any given level $n\in\N$ for the Sierpi\'nski graphs. See Figure \ref{fig:probabilities_sim} for the calculations of level $n=4$. 

\begin{figure}
    \centering
    \include{gasket_height_tikz}
    \caption{Probabilities for the number of neighbours that are descendants in a two-component forest of $\SG_4$, where the right and upper corners are in distinct connected components.}
    \label{fig:probabilities_sim}
\end{figure}

\section{Expected height}\label{sec:expHeight}
This section is devoted to calculating the expected height of a sandpile as well as the expected number of vertices of height $i$ for $i\in\{0,1,2,3\}$ of a sandpile sampled from the stationary distribution of the Abelian sandpile model on $\SG_n$.
Denote by $\lcorner{n},\tcorner{n},\rcorner{n}$ the corner vertices of $\SG_n$. For a sandpile configuration $\sigma$, we define the total weight of the sandpile by
\begin{align*}
    W_n(\sigma)=\sum_{v\in\SG_n\backslash\{\lcorner{n},\rcorner{n},\tcorner{n}\}}\sigma(v),
\end{align*}
and the weight of the number of vertices of height $i$ for $i\in\{0,1,2,3\}$
\begin{align*}
    W_n^i(\sigma)=\sum_{v\in\SG_n\backslash\{\lcorner{n},\rcorner{n},\tcorner{n}\}}\delta_i(\sigma(v)).
\end{align*}
We use again the burning bijection to derive expressions for the expectations of $W_n$ and $W_n^i$ based on the average number of neighbours that are descendants of each vertex. 
Given a forest $T\in\mathcal{Q}_n=\mathcal{T}_n\cup \mathcal{S}_n^1\cup\mathcal{S}_n^2\cup\mathcal{S}_n^3\cup\mathcal{R}_n$, we define the total number of descendants of $T$ by
\begin{align*}
    D_n(T)=\sum_{v\in\SG_n\backslash\{\lcorner{n},\rcorner{n},\tcorner{n}\}}\desc{n}{T}{v},
\end{align*}
as well as the total number of vertices in $\SG_n$ that have $i$ neighbours as descendants for $i\in\{0,1,2,3\}$ by
\begin{align*}
    D_n^i(T)=\sum_{v\in\SG_n\backslash\{\lcorner{n},\rcorner{n},\tcorner{n}\}}\delta_i(\desc{n}{T}{v}).
\end{align*}
In order to simplify the computations of the expectations, we also introduce the following notation
\begin{align*}
    \overline{D}_n(\onecomp)&=\mathbb{E}\big[D_n(T)|T\in\mathcal{T}_n\big],\\
    \overline{D}_n(\twocompleft)&=\mathbb{E}\big[D_n(T)|T\in\mathcal{S}_n^1\big],\\
    \overline{D}_n(\twocompup)&=\mathbb{E}\big[D_n(T)|T\in\mathcal{S}_n^2\big],\\
    \overline{D}_n(\twocompright)&=\mathbb{E}\big[D_n(T)|T\in\mathcal{S}_n^3\big],\\
    \overline{D}_n(\threecomp)&=\mathbb{E}\big[D_n(T)|T\in\mathcal{R}_n\big].
\end{align*}
Similarly, we introduce the notation for the expected number of vertices in $\SG_n$ that have $i$ neighbours as descendants for $i\in\{0,1,2,3\}$ as
\begin{align*}
    \overline{D}_n^i(\onecomp)&=\mathbb{E}\big[D_n^i(T)|T\in\mathcal{T}_n\big],\\
    \overline{D}_n^i(\twocompleft)&=\mathbb{E}\big[D_n^i(T)|T\in\mathcal{S}_n^1\big],\\
    \overline{D}_n^i(\twocompup)&=\mathbb{E}\big[D_n^i(T)|T\in\mathcal{S}_n^2\big],\\
    \overline{D}_n^i(\twocompright)&=\mathbb{E}\big[D_n^i(T)|T\in\mathcal{S}_n^3\big],\\
    \overline{D}_n^i(\threecomp)&=\mathbb{E}\big[D_n^i(T)|T\in\mathcal{R}_n\big].
\end{align*}
It holds
\begin{align} \label{eq:expected-desc}
    \overline{D}_n(\onecomp)=\sum_{i=0}^3 i\overline{D}_n^i(\onecomp),
\end{align}
and similarly for the other forests on $\SG_n$. This formula is easily obtained by plugging in the definition of expectation for discrete random variables. Our goal is to calculate $\overline{D}_n^i$ for the different component forests we have on $\SG_n$. We will make use of the recursive structure of forests on $\SG_n$ as described in \cite{ust-on-gasket} and once again in Figure \ref{fig:recursive_spanningTree}, \ref{fig:recursive_twoComponent} and \ref{fig:recursive_threeComponent}. Noticing that the number of descendants of a vertex that is not a cut point in a forest $T$ on $\SG_n$ is the same as the number of neighbours that are descendants in the forest of the smaller subtriangle, we obtain the recursion
\begin{equation}
\begin{aligned}\label{eq:recursion-i-descendants}
    \begin{pmatrix}
    \overline{D}_n^i(\onecomp)\\
    \overline{D}_n^i(\twocompleft)\\
    \overline{D}_n^i(\twocompup)\\
    \overline{D}_n^i(\twocompright)\\
    \overline{D}_n^i(\threecomp)
    \end{pmatrix}
    =&\frac{1}{150}\begin{pmatrix}
    300 & 50 & 50 & 50 & 0\\
    195 & 150 & 30 & 30 & 45\\
    195 & 30 & 150 & 30 & 45\\
    195 & 30 & 30 & 150 & 45\\
    108 & 78 & 78 & 78 & 108
    \end{pmatrix}\cdot
    \begin{pmatrix}
    \overline{D}_{n-1}^i(\onecomp)\\
    \overline{D}_{n-1}^i(\twocompleft)\\
    \overline{D}_{n-1}^i(\twocompup)\\
    \overline{D}_{n-1}^i(\twocompright)\\
    \overline{D}_{n-1}^i(\threecomp)
    \end{pmatrix}
    \\&+\sum_{j=1}^3 \begin{pmatrix}
    \mathbb{P}(\desc{n}{T}{\cutpointsymbol_j^n}=i|T\in\mathcal{T}_n)\\
    \mathbb{P}(\desc{n}{T}{\cutpointsymbol_j^n}=i|T\in\mathcal{S}_n^1)\\
    \mathbb{P}(\desc{n}{T}{\cutpointsymbol_j^n}=i|T\in\mathcal{S}_n^2)\\
    \mathbb{P}(\desc{n}{T}{\cutpointsymbol_j^n}=i|T\in\mathcal{S}_n^3)\\
    \mathbb{P}(\desc{n}{T}{\cutpointsymbol_j^n}=i|T\in\mathcal{R}_n)
    \end{pmatrix},
\end{aligned}
\end{equation}
where $\cutpointsymbol_1^n,\cutpointsymbol_2^n,\cutpointsymbol_3^n$ are the cutpoints in $\SG_n$; see Figure \ref{fig:first-3-it-SG}. The matrix
\begin{align*}
    M:=\begin{pmatrix}
    300 & 50 & 50 & 50 & 0\\
    195 & 150 & 30 & 30 & 45\\
    195 & 30 & 150 & 30 & 45\\
    195 & 30 & 30 & 150 & 45\\
    108 & 78 & 78 & 78 & 108
    \end{pmatrix}
\end{align*}
can be diagonalized and its eigenvalues are given by
\begin{align*}
    \lambda_1=450,\lambda_2=150,\lambda_2=120,\lambda_2=120,\lambda_2=18,
\end{align*}
while the corresponding eigenvectors are
\begin{align*}
    &v_1=(1,1,1,1,1),\ 
    v_2=(-1,1,1,1,3),\\
    &v_3=(0,-1,0,1,0),\ 
    v_4=(0,-1,1,0,0),\\
    &v_5=(125,-235,-235,-235,461).
\end{align*}
Let us now define the expected neighbours that are descendants of the cut points, as in Equation (\ref{eq:recursion-i-descendants}) in the second line, by
\begin{align*}
    \begin{pmatrix}
    e_n(\onecomp,i)\\
    e_n(\twocompleft,i)\\
    e_n(\twocompup,i)\\
    e_n(\twocompright,i)\\
    e_n(\threecomp,i)
    \end{pmatrix}
    =
    \sum_{j=1}^3 \begin{pmatrix}
    \mathbb{P}(\desc{n}{T}{\cutpointsymbol_j^n}=i|T\in\mathcal{T}_n)\\
    \mathbb{P}(\desc{n}{T}{\cutpointsymbol_j^n}=i|T\in\mathcal{S}_n^1)\\
    \mathbb{P}(\desc{n}{T}{\cutpointsymbol_j^n}=i|T\in\mathcal{S}_n^2)\\
    \mathbb{P}(\desc{n}{T}{\cutpointsymbol_j^n}=i|T\in\mathcal{S}_n^3)\\
    \mathbb{P}(\desc{n}{T}{\cutpointsymbol_j^n}=i|T\in\mathcal{R}_n)
    \end{pmatrix}.
\end{align*}
We then rewrite Equation (\ref{eq:recursion-i-descendants}) by repeatedly applying the recursion to all the terms of the form $\overline{D}_n^i(\cdot)$ to obtain
\begin{align*}
    \begin{pmatrix}
    \overline{D}_n^i(\onecomp)\\
    \overline{D}_n^i(\twocompleft)\\
    \overline{D}_n^i(\twocompup)\\
    \overline{D}_n^i(\twocompright)\\
    \overline{D}_n^i(\threecomp)
    \end{pmatrix}
    =&\frac{M}{150}
    \begin{pmatrix}
    \overline{D}_{n-1}^i(\onecomp)\\
    \overline{D}_{n-1}^i(\twocompleft)\\
    \overline{D}_{n-1}^i(\twocompup)\\
    \overline{D}_{n-1}^i(\twocompright)\\
    \overline{D}_{n-1}^i(\threecomp)
    \end{pmatrix}+
    \begin{pmatrix}
    e_n(\onecomp,i)\\
    e_n(\twocompleft,i)\\
    e_n(\twocompup,i)\\
    e_n(\twocompright,i)\\
    e_n(\threecomp,i)
    \end{pmatrix}
    =\sum_{j=0}^{n-1}\frac{M^j}{150^{j}}\begin{pmatrix}
    e_{n-j}(\onecomp,i)\\
    e_{n-j}(\twocompleft,i)\\
    e_{n-j}(\twocompup,i)\\
    e_{n-j}(\twocompright,i)\\
    e_{n-j}(\threecomp,i)
    \end{pmatrix}
\end{align*}
In the previous equation, we can rewrite the powers of $M$ using its eigenvalue decomposition as
$$M^j=S\begin{pmatrix}
18^j&0&0&0&0\\
0&120^j&0&0&0\\
0&0&120^j&0&0\\
0&0&0&150^j&0\\
0&0&0&0&450^j
\end{pmatrix}S^{-1},$$
where $S$ is the matrix whose columns are given as the eigenvectors of $M$. Using the matrix diagonalization of $M$ and plugging in the results on $e_n$ from Section \ref{subsection:cut-points} (see also the appendix for a closed form expression of $e_n$), we obtain the limits
\begin{align*}
    &\phantom{=\ }\lim_{n\rightarrow\infty}\frac{1}{|\SG_n|}\begin{pmatrix}
    \overline{D}_n^0(\onecomp)\\\overline{D}_n^1(\onecomp)\\\overline{D}_n^2(\onecomp)\\\overline{D}_n^3(\onecomp)
    \end{pmatrix}=\lim_{n\rightarrow\infty}\frac{1}{|\SG_n|}\begin{pmatrix}
    \overline{D}_n^0(\twocompleft)\\\overline{D}_n^1(\twocompleft)\\\overline{D}_n^2(\twocompleft)\\\overline{D}_n^3(\twocompleft)
    \end{pmatrix}=\lim_{n\rightarrow\infty}\frac{1}{|\SG_n|}\begin{pmatrix}
    \overline{D}_n^0(\twocompup)\\\overline{D}_n^1(\twocompup)\\\overline{D}_n^2(\twocompup)\\\overline{D}_n^3(\twocompup)
    \end{pmatrix}=\\&=\lim_{n\rightarrow\infty}\frac{1}{|\SG_n|}\begin{pmatrix}
    \overline{D}_n^0(\twocompright)\\\overline{D}_n^1(\twocompright)\\\overline{D}_n^2(\twocompright)\\\overline{D}_n^3(\twocompright)
    \end{pmatrix}=\lim_{n\rightarrow\infty}\frac{1}{|\SG_n|}\begin{pmatrix}
    \overline{D}_n^0(\threecomp)\\\overline{D}_n^1(\threecomp)\\\overline{D}_n^2(\threecomp)\\\overline{D}_n^3(\threecomp)
    \end{pmatrix}=\begin{pmatrix}10957/40464\\22737599/87978852\\33273907/87978852\\3619595/39101712\end{pmatrix}\approx\begin{pmatrix}0.27\\0.25\\0.38\\0.10\end{pmatrix},
\end{align*}
and thus, using the relation between $\overline{D}_n^i(\cdot)$ for all $i\in\{0,1,2,3\}$ and the average height $\overline{D}_n(\cdot)$ from Equation (\ref{eq:expected-desc}), we  get
\begin{align*}
    \lim_{n\rightarrow\infty}\frac{1}{|\SG_n|}\begin{pmatrix}
    \overline{D}_n(\onecomp)\\
    \overline{D}_n(\twocompleft)\\
    \overline{D}_n(\twocompup)\\
    \overline{D}_n(\twocompright)\\
    \overline{D}_n(\threecomp)
    \end{pmatrix}=\lim_{n\rightarrow\infty}\frac{1}{|\SG_n|}\sum_{i=0}^3 i \begin{pmatrix}
    \overline{D}_n^i(\onecomp)\\
    \overline{D}_n^i(\twocompleft)\\
    \overline{D}_n^i(\twocompup)\\
    \overline{D}_n^i(\twocompright)\\
    \overline{D}_n^i(\threecomp)
    \end{pmatrix}=\frac{7259}{5616}\begin{pmatrix}1\\1\\1\\1\\1\end{pmatrix}\approx\begin{pmatrix}1.3\\1.3\\1.3\\1.3\\1.3\end{pmatrix}.
\end{align*}
{We collect the exact values of $\overline{D}_n(\onecomp),\overline{D}_n(\twocompleft),\overline{D}_n(\twocompup), \overline{D}_n(\twocompright),  \overline{D}_n(\threecomp)$ in Appendix \ref{sec:appendixComp}.}
Denote by $T(\sigma)$ the spanning tree of $\SG_n$ obtained by applying to $\sigma$ the burning algorithm. Denote by $\overline{W}_n^i(\onecomp)$ the expectation of $W_n^i$ taken over the set of recurrent sandpiles with sink given by $\tcorner{n}$ with the uniform measure. We can then obtain an expression for $\overline{W}_n^i(\onecomp)$ in terms of $\overline{D}_n^i(\onecomp)$ for all $i\in\{0,1,2,3\}$ similarly to Equation (\ref{eq:expected-desc}) by employing Lemma \ref{lem:desc-conn-to-height}:
\begin{align*}
    \overline{W}_n^i(\onecomp)&=\sum_{v\in\SG_n\backslash\{\lcorner{n},\rcorner{n},\tcorner{n}\}}\mathbb{P}(\sigma(v)=i)\\
    &=\sum_{v\in\SG_n\backslash\{\lcorner{n},\rcorner{n},\tcorner{n}\}}\sum_{j=0}^i\mathbb{P}(\sigma(v)=i|\desc{n}{T(\sigma)}{v}=j)\mathbb{P}(\desc{n}{T(\sigma)}{v}=j)\\
    &=\sum_{j=0}^i \frac{1}{4-j}\overline{D}_n^j(\onecomp).
\end{align*}
Let us further denote by $\overline{W}_n^i(\twocompright\cup\twocompup)$ the expectation of $W_n^i$ taken over the set of recurrent sandpiles with $\rcorner{n}$ and $\tcorner{n}$ as sinks, and by $\overline{W}_n^i(\threecomp)$ the expectation of $W_n^i$ taken over the set of recurrent sandpiles with all corners as sinks. Then in the same fashion we obtain the equations
\begin{align*}
    &\overline{W}_n^i(\twocompright\cup\twocompup)=\sum_{j=0}^i\frac{1}{4-j}\frac{\overline{D}_n^j(\twocompright)+\overline{D}_n^j(\twocompup)}{2},\\
    &\overline{W}_n^i(\threecomp)=\sum_{j=0}^i\frac{1}{4-j}\overline{D}_n^j(\threecomp).
\end{align*}
We can now use these relations to obtain the limit for all $\overline{W}_n^i(\cdot)$ for all $i\in\{0,1,2,3\}$ as
\begin{equation}
\begin{aligned}\label{eq:average-vertices-with-height}
    \lim_{n\rightarrow\infty}\frac{1}{|\SG_n|}\begin{pmatrix}
    \overline{W}_n^0(\onecomp)\\\overline{W}_n^1(\onecomp)\\\overline{W}_n^2(\onecomp)\\\overline{W}_n^3(\onecomp)\end{pmatrix}=\lim_{n\rightarrow\infty}\frac{1}{|\SG_n|}\begin{pmatrix}
    \overline{W}_n^0(\twocompright\cup\twocompup)\\\overline{W}_n^1(\twocompright\cup\twocompup)\\\overline{W}_n^2(\twocompright\cup\twocompup)\\\overline{W}_n^3(\twocompright\cup\twocompup)\end{pmatrix}=&\\\lim_{n\rightarrow\infty}\frac{1}{|\SG_n|}\begin{pmatrix}
    \overline{W}_n^0(\threecomp)\\\overline{W}_n^1(\threecomp)\\\overline{W}_n^2(\threecomp)\\\overline{W}_n^3(\threecomp)\end{pmatrix}=\begin{pmatrix}
        10957/161856\\649680671/4222984896\\1448254439/4222984896\\1839170699/4222984896
    \end{pmatrix}\approx\begin{pmatrix}
        0.07\\0.15\\0.34\\0.44
    \end{pmatrix}.
\end{aligned}
\end{equation}
If we now define $\overline{W}_n(\cdot)$ similarly to $\overline{D}_n(\cdot)$ as
\begin{align*}
&\overline{W}_n(\onecomp)=\mathbb{E}[W_n(\sigma)|T(\sigma)\in\onecomp]\\
&\overline{W}_n(\twocompright\cup\twocompup)=\mathbb{E}[W_n(\sigma)|
T(\sigma)\in\twocompright\cup\twocompup]\\
&\overline{W}_n(\threecomp)=\mathbb{E}[W_n(\sigma)|T(\sigma)\in\threecomp],
\end{align*}
we obtain the limit of the expected average height for sandpiles with the three different choices of the sink vertex by using Equation (\ref{eq:average-vertices-with-height})
\begin{align*}
    \lim_{n\rightarrow\infty}\frac{1}{|\SG_n|}\begin{pmatrix}
    \overline{W}_n(\onecomp)\\
    \overline{W}_n(\twocompright\cup\twocompup)\\
    \overline{W}_n(\threecomp)
    \end{pmatrix}=\lim_{n\rightarrow\infty}\frac{1}{|\SG_n|}\sum_{i=0}^3 i \begin{pmatrix}
    \overline{W}_n^i(\onecomp)\\
    \overline{W}_n^i(\twocompright\cup\twocompup)\\
    \overline{W}_n^i(\threecomp)
    \end{pmatrix}=\frac{24107}{11232}
    \begin{pmatrix}
        1 \\ 1 \\ 1
    \end{pmatrix}\approx\begin{pmatrix}
        2.15\\2.15\\2.15
    \end{pmatrix}
\end{align*}
A curious observation to be made after obtaining the average height of a recursive sandpile on the Sierpinski gasket graphs is that in the limit, as we sent the number of iterations to infinity, the average height does not  depend on the choice and the number of sink vertices. This is at first sight rather counter-intuitive, but we first want to emphasize that for all $n\in\N$, on the finite iteration graph of level $n$, the choice of the sink vertex does indeed change the value of the average height as well as the height probabilities; see the appendix for closed form expressions. When considering the recursive construction of spanning trees and forests on the finite iteration graphs in \cite{ust-on-gasket}, as trees and forests are made up of trees and forests of lower iteration gaskets, with the number of iterations going to infinity, certain statistics of spanning trees - such as the average number of neighbours that are descendants - get mixed and the number of components of the spanning forest is forgotten. This property then carries over to sandpiles and to the choice of the sink vertex via Lemma \ref{lem:desc-conn-to-height}. It would be interesting to understand if there are other statistics that cannot remember the sink vertices in the limit, or to consider graphs other than the Sierpinski gasket with different choices of sink vertices.

\section{Connection to the looping constant}\label{sec:loop-cst}

The final part is devoted to showing a connection between the average weight (or height) of the recurrent sandpiles and the expected number of neighbours of the starting vertex in a loop erased random walk. We want to emphasize  that this connection was already known. Our contribution here is the calculation of the looping constant on the Sierpinski gasket graph explicitly, using \cite{loop-conn, looping-constant-of-zd} together with our results from the previous sections. Our result concerning the looping constant differs from  \cite{loop-conn, looping-constant-of-zd} in the sense that we show a correspondence between the bulk average height and the average looping constant, where the average is taken over all vertices. This is because the Sierpinski gasket is not translation invariant, and thus the looping constant is different at different vertices.

\textbf{Loop erased random walk.}
Let $G$ be any connected graph and consider a finite path $\gamma = (x_1,\ldots,x_n)$ of length $ n\in \N$ in $G$. Define inductively: $i_1=1$,  and for $j>1$
\begin{align*}
i_j=\max\{i\leq n:x_i=x_{i_{j-1}}\}+1.
\end{align*}
The induction stops when for some $J\in \N$ we have $x_{i_J}=x_n$. We define the loop erasure of $\gamma$ as
\begin{align*}
    \LE(\gamma)=(x_{i_1},...,x_{i_J}),
\end{align*}
which is the path obtained by consecutively deleting cycles in the path $\gamma$.

Consider the Sierpi\'nski graph $\SG_n$ of level $n$ and let $C_n = \{\tcorner{n},\rcorner{n}\}$ be the right and upper corners of $\SG_n$. Let $v$ be an arbitrary vertex in $\SG_n$ and let $X_v$ be the simple random walk on $\SG_n$ started in $v$ and stopped when first visiting $C_n$. The loop-erased random walk started at $v$ is defined to be the random path with distribution given by $\LE(X_v)$. Finally, the \emph{looping constant} at vertex $v$ is defined as
\begin{align*}
    \zeta_v=\mathbb{E}\big[|\{\text{neighbours of $v$ visited by $\LE(X_v)$}\}|\big],
\end{align*}
and similarly to the heights in recurrent sandpiles, we define the bulk average looping constant by
\begin{align*}
    \zeta_n=\frac{1}{|\SG_n|}\sum_{v\in \SG_n}\zeta_v.
\end{align*}
It is a well-known fact that the path from $v$ to the roots in a uniform spanning forest with root set given by $C_n$ has as distribution the loop-erased random walk as defined above. This fact can for example be found in \cite[Chapter 4.1]{loop-walk}. We consider the bulk averages on $\SG_n$ and we prove a connection between the looping constant and expected heights of recurrent sandpiles.

For any rooted tree $T$ and two vertices $v$ and $w$, we say that $w$ is an ascendant of $v$ in $T$, if the unique path from $v$ to the root of $T$ passes along $w$, shortly $v <_T w$. Then we can rewrite
\begin{align}
    \zeta_v=\sum_{w\sim v}\mathbb{P}(v<_{T^{(2)}}w),
\end{align}
where $T^{(2)}$ is a uniformly distributed two-component forest on $\SG_n$ with root set given by $C_n$, and  this expression is similar to the expected number of descendants as calculated previously. For
\begin{align*}
    \xi_n=\frac{1}{|\SG_n|}\sum_{v\in\SG_n}{\E\big[\desc{n}{T}{v} \ | \ t \in \twoCompSetTwo{n} \cup \twoCompSetThree{n} \big] },
\end{align*}
we have
\begin{align*}
    \xi=\lim_{n\rightarrow\infty}\xi_n.
\end{align*}     
Both $\xi$ and $\xi_n$ for $n\in\N$ have been calculated before in Section \ref{sec:expHeight}, where $\xi_n$ is given by $\frac{\overline{D}_n(\twocompright)+\overline{D}_n(\twocompup)}{2}$ plus the expectation at the corner vertices. Moreover  $\zeta_n$ can also be rewritten as:
\begin{align}\label{eq:loop-const-with-edges}
    \zeta_n=\frac{1}{|\SG_n|}\sum_{v\in\SG_n}\sum_{w\sim v}\mathbb{P}(v<_{T^{(2)}} w)=\frac{1}{|\SG_n|}\sum_{\{x,y\}\in E_n}\mathbb{P}(x<_{T^{(2)}} y)+\mathbb{P}(y<_{T^{(2)}} x).
\end{align}
This observation will be used below, where we show that the bulk average number of descendants converges to the same value as the bulk average looping constant.
\begin{lem}\label{lem:loop=desc}
On $\SG_n$ we have
\begin{align*}
    \lim_{n\rightarrow\infty}\zeta_n=\xi.
\end{align*}
\end{lem}
\begin{proof}
We have
\begin{align*}
    \xi_n=\frac{1}{|\SG_n|}\sum_{v\in\SG_n}\sum_{w\sim v}\mathbb{P}(w<_{T^{(2)}}v)=\frac{1}{|\SG_n|}\sum_{\{x,y\}\in E_n}\mathbb{P}(x<_{T^{(2)}} y)+\mathbb{P}(y<_{T^{(2)}} x)
\end{align*}
which together with  Equation \eqref{eq:loop-const-with-edges} yields $\xi_n=\zeta_n$,
and thus $\lim_{n\rightarrow\infty}\zeta_n=\lim_{n\rightarrow\infty}\xi_n=\xi$
and this proves the claim.
\end{proof}
Hence, using the calculations in Section \ref{sec:expHeight}, we obtain the value of the  looping constant as
\begin{align*}
    \zeta=\frac{7259}{5616}.
\end{align*}
We can finally show the connection between the bulk average sandpile height and the bulk average looping constant.

\begin{prop}\label{thm:loop-rel-to-height}
If the bulk average height is given by
\begin{align*}
    \sigma= \lim_{m\rightarrow\infty}\frac{1}{|\SG_n|}\sum_{v\in \SG_n}\mathbb{E}\big[\sigma(v)\big],
\end{align*}
then we have
\begin{align*}
    \sigma=\frac{\zeta+3}{2}.
\end{align*}
\end{prop}
\begin{proof}
It holds
\begin{align*}
    \sigma = \lim_{m\rightarrow\infty}\frac{1}{|\SG_n|}\sum_{v\in \SG_n}\frac{\mathbb{E}\big[\desc{n}{t}{v}\big]+\deg(v)-1}{2}
    =\frac{\zeta+3}{2},
\end{align*}
where the first equality follows from \cite[Lemma 8]{looping-constant-of-zd} and the second one from Lemma \ref{lem:loop=desc}.
\end{proof}

\section{Outlook and related research questions}

Our calculations and results have been made on finite  Sierpi\'nski graphs, but it is natural to ask what happens on the infinite  Sierpi\'nski graph. Does the stationary distribution (i.e. the uniform measure on recurrent configurations) of the sandpile Markov chain on $\SG_n$ converge weakly to a measure supported on the infinite Sierpi\'nski graph? The existence and uniqueness of such a measure, called \emph{the uniform volume limit measure of sandpiles} follows from the fact that  the uniform spanning tree on the infinite Sierpi\'nski graph is one-ended almost surely; see \cite{sierp-one-end,engelenburg-berestycki, engelenburg-hutchcroft} for general graphs and \cite{inf-vol-limit-sandpile} for $\Z^2$. Our results about heights and expected height can be extended to the infinite volume setting. Another interesting statistic in the context of sandpiles is the distribution of waves and avalanches during stabilization in infinite volume. If $\sigma$ is a sandpile sampled from the infinite volume measure, does $\sigma+\delta_o$ stabilize almost surely, and if so, can we describe the distribution of the avalanche, that is
$$\mathbb{P}(|\{\text{vertices toppled during stabilization of }\sigma+\delta_o\}|>R)$$ for $R\in\N$? It is believed and supported by simulations \cite{physics3} that the size of avalanches on infinite gaskets follows a power law, that is, there exists $\gamma\in(0,\infty)$ such that for all $R\in\N$ we have
\begin{align*}
    \mathbb{P}(|\{\text{vertices toppled during stabilization of }\sigma+\delta_o\}|>R)\sim R^{-\gamma}.
\end{align*}
On $\mathbb{Z}^d$, for  $d\geq3$ it has been shown in \cite{crit-exp-z^d}, that there exist $\gamma_1,\gamma_2\in(0,\infty)$ and constants $C_1,C_2>0$ such that for all $R\in\N$ we have
\begin{align*}
C_1\cdot R^{-\gamma_1}\leq \mathbb{P}(|\{\text{vertices toppled during stabilization of }\sigma+\delta_o\}|>R)\leq C_2\cdot R^{-\gamma_2}.
\end{align*}
In dimension $2$, only a lower bound has been proven.
Some bounds on the avalanche size can be given on the infinite Sierpi\'nski graph by exploiting the recursive structure of the spanning trees on it.
Another interesting question on the Sierpi\'nski gasket graphs is to study the recursive structure of recurrent sandpiles on $\SG_n$.

\appendix

\section{Collected computational results}\label{sec:appendixComp}

The probability for $k$ descendants of roots denoted as $\nRole{n}{2}, \noRole{n}{2}, \nRole{n}{3}$ defined in Section \ref{subsec:neighborsSameComp} are given by the values depicted in Table \ref{tab:rootProbs} .


Let $$m_n = \left(1,\left( \frac{3}{5} \right)^n, \left( \frac{2}{5} \right)^n, \left( \frac{1}{15} \right)^n, \left( \frac{1}{25} \right)^n, \left( \frac{2}{75} \right)^n, \left( \frac{1}{375} \right)^n \right). $$
Then the various probabilities of the cut points for the corresponding forests as elaborated in Section \ref{subsection:cut-points} are given by \\
\makebox[\textwidth][c]{
$
\begin{pmatrix}
	\Prob(\desc{n}{T}{\bcut{n}} = 0 | T \in \mathcal{T}_n) \\
	\Prob(\desc{n}{T}{\bcut{n}} = 1 | T \in \mathcal{T}_n) \\
	\Prob(\desc{n}{T}{\bcut{n}} = 2 | T \in \mathcal{T}_n)\\
	\Prob(\desc{n}{T}{\bcut{n}} = 3 | T \in \mathcal{T}_n)
\end{pmatrix}= \begin{pmatrix}
	0 & \phantom{-} \frac{605}{1176} & \phantom{-} 0 & \phantom{-} 0 & -\frac{1 375}{588} & \phantom{-} 0 & \phantom{-} \frac{3 125}{1 176} \\[3pt] 
	0 & \phantom{-} \frac{110}{147} & - \frac{605}{1 512} & \phantom{-} 0 & \phantom{-} \frac{2 375}{2 646} & \phantom{-} \frac{1 375}{1512} & - \frac{15 625}{2 646} \\[3pt]
	\frac{11}{14} & - \frac{375}{392} & \phantom{-} \frac{55}{189} & - \frac{25}{14} & \phantom{-} \frac{5 375}{1 323} & - \frac{1 375}{756} & \phantom{-} \frac{40 625}{10 584} \\[3pt]
	\frac{3}{14} & - \frac{15}{49} & \phantom{-} \frac{55}{504} & \phantom{-} \frac{25}{14} & - \frac{4 625}{1 764} &  \phantom{-} \frac{1 375}{1 512} & - \frac{3 125}{5 292}
\end{pmatrix}m_n^T,$}\\[5pt] \makebox[\textwidth][c]{
$\begin{pmatrix}
	\Prob(\desc{n}{T}{\lcut{n}} = 0 | T \in \mathcal{T}_n) \\
	\Prob(\desc{n}{T}{\lcut{n}} = 1 | T \in \mathcal{T}_n) \\
	\Prob(\desc{n}{T}{\lcut{n}} = 2 | T \in \mathcal{T}_n) \\
	\Prob(\desc{n}{T}{\lcut{n}} = 3 | T \in \mathcal{T}_n)
\end{pmatrix}= \begin{pmatrix}
	 0 & \phantom{-} \frac{605}{1176} & \phantom{-} 0 & \phantom{-} 0 & {-} \frac{1375}{588} & \phantom{-} 0 & \phantom{-} \frac{3125}{1176} \\[3pt]
	 0 & \phantom{-} \frac{110}{147} & {-} \frac{275}{378} & \phantom{-} 0 & \phantom{-} \frac{2375}{2646} & \phantom{-} \frac{625}{378} & {-} \frac{15625}{2646} \\[3pt]
	\frac{11}{14} & {-} \frac{375}{392} & \phantom{-} \frac{100}{189} & {-} \frac{20}{21} & \phantom{-} \frac{5375}{1323} & {-} \frac{625}{189} & \phantom{-} \frac{40625}{10584} \\[3pt]
	\frac{3}{14} & {-} \frac{15}{49} & \phantom{-} \frac{25}{126} & \phantom{-} \frac{20}{21} & {-} \frac{4625}{1764} & \phantom{-} \frac{625}{378} & {-} \frac{3125}{5292} 
\end{pmatrix}m_n^T,$}\\[5pt] \makebox[\textwidth][c]{
$\begin{pmatrix}
	\Prob(\desc{n}{T}{\bcut{n}} = 0 | T \in \mathcal{S}^{2}_n) \\
	\Prob(\desc{n}{T}{\bcut{n}} = 1 | T \in \mathcal{S}^{2}_n) \\
	\Prob(\desc{n}{T}{\bcut{n}} = 2 | T \in \mathcal{S}^{2}_n)\\
	\Prob(\desc{n}{T}{\bcut{n}} = 3 | T \in \mathcal{S}^{2}_n)
\end{pmatrix}= \begin{pmatrix}
	 0 & \phantom{-} \frac{121}{392} & \phantom{-} 0 & \phantom{-} 0 & {-} \frac{275}{196} & \phantom{-} 0 & \phantom{-} \frac{625}{392} \\[3pt]
	 0 & \phantom{-} \frac{22}{49} & {-} \frac{11}{252} & \phantom{-} 0 & \phantom{-} \frac{475}{882} & \phantom{-} \frac{85}{63} & {-} \frac{3125}{882} \\[3pt]
	 \frac{11}{14} & {-} \frac{225}{392} & \phantom{-} \frac{2}{63} & {-} \frac{2}{7} & \phantom{-} \frac{1 075}{441} & {-} \frac{170}{63} & \phantom{-} \frac{8125}{3528} \\[3pt]
	\frac{3}{14} & {-} \frac{9}{49} & \phantom{-} \frac{1}{84} & \phantom{-} \frac{2}{7} & {-} \frac{925}{588} & \phantom{-} \frac{85}{63} & {-} \frac{625}{1764} 
\end{pmatrix}m_n^T,$}\\[5pt] \makebox[\textwidth][c]{
$\begin{pmatrix}
	\Prob(\desc{n}{T}{\rcut{n}} = 0 | T \in \mathcal{S}^{2}_n) \\
	\Prob(\desc{n}{T}{\rcut{n}} = 1 | T \in \mathcal{S}^{2}_n) \\
	\Prob(\desc{n}{T}{\rcut{n}} = 2 | T \in \mathcal{S}^{2}_n)\\
	\Prob(\desc{n}{T}{\rcut{n}} = 3 | T \in \mathcal{S}^{2}_n)
\end{pmatrix}= \begin{pmatrix}
	0 & \phantom{-} \frac{363}{392} & \phantom{-} 0 & \phantom{-} 0 & {-} \frac{110}{392} & \phantom{-} 0 & {-} \frac{1625}{392} \\[3pt]
	0 & \phantom{-} \frac{66}{49} & {-} \frac{77}{72} & \phantom{-} 0 & \phantom{-} \frac{95}{882} & {-} \frac{25}{72} & \phantom{-} \frac{8125}{882} \\[3pt]
	\frac{11}{14} & {-} \frac{675}{392} & \phantom{-} \frac{7}{9} & {-} \frac{2}{7} & \phantom{-} \frac{215}{441} & \phantom{-} \frac{25}{36} & {-} \frac{21125}{3528} \\[3pt]
	\frac{3}{14} & {-} \frac{27}{49} & \phantom{-} \frac{7}{24} & \phantom{-} \frac{2}{7} & {-} \frac{185}{588} & {-} \frac{25}{72} & \phantom{-} \frac{1625}{1764} 
\end{pmatrix}m_n^T,$}\\[5pt] \makebox[\textwidth][c]{
$\begin{pmatrix}
	\Prob(\desc{n}{T}{\lcut{n}} = 0 | T \in \mathcal{S}^{2}_n) \\
	\Prob(\desc{n}{T}{\lcut{n}} = 1 | T \in \mathcal{S}^{2}_n) \\
	\Prob(\desc{n}{T}{\lcut{n}} = 2 | T \in \mathcal{S}^{2}_n)\\
	\Prob(\desc{n}{T}{\lcut{n}} = 3 | T \in \mathcal{S}^{2}_n)
\end{pmatrix}= \begin{pmatrix}
	0 & \phantom{-} \frac{363}{392} & \phantom{-} 0 & \phantom{-} 0 & {-} \frac{110}{392} & \phantom{-} 0 & {-} \frac{1625}{392} \\[3pt]
	0 & \phantom{-} \frac{66}{49} & {-} \frac{319}{252} & \phantom{-} 0 & \phantom{-} \frac{95}{882} & \phantom{-} \frac{25}{252} & \phantom{-} \frac{8125}{882} \\[3pt]
	\frac{11}{14} & {-} \frac{675}{392} & \phantom{-} \frac{58}{63} & \phantom{-} \frac{3}{14} & \phantom{-} \frac{215}{441} & {-} \frac{25}{126} & {-} \frac{21125}{3528} \\[3pt]
	\frac{3}{14} & {-} \frac{27}{49} & \phantom{-} \frac{29}{84} & {-} \frac{3}{14} & {-} \frac{185}{588} & \phantom{-} \frac{25}{252} & \phantom{-} \frac{1625}{1764} 
\end{pmatrix}m_n^T,$}\\[5pt] \makebox[\textwidth][c]{
$\begin{pmatrix}
	\Prob(\desc{n}{T}{\bcut{n}} = 0 | T \in \mathcal{R}^{1}_n) \\
	\Prob(\desc{n}{T}{\bcut{n}} = 1 | T \in \mathcal{R}^{1}_n) \\
	\Prob(\desc{n}{T}{\bcut{n}} = 2 | T \in \mathcal{R}^{1}_n) \\
	\Prob(\desc{n}{T}{\bcut{n}} = 3 | T \in \mathcal{R}^{1}_n)
\end{pmatrix}= \begin{pmatrix}
	0 & \phantom{-} \frac{363}{392} & \phantom{-} 0 & \phantom{-} 0 & \phantom{-} \frac{814}{392} & \phantom{-} 0 & \phantom{-} \frac{195}{392} \\[3pt]
	0 & \phantom{-} \frac{66}{49} & {-} \frac{2629}{2520} & \phantom{-} 0 & {-} \frac{703}{882} & {-} \frac{227}{168} & {-} \frac{325}{294} \\[3pt]
	\frac{11}{14} & {-} \frac{675}{392} & \phantom{-} \frac{239}{315} & \phantom{-} \frac{57}{70} & {-} \frac{1591}{441} & \phantom{-} \frac{227}{84} & \phantom{-} \frac{845}{1176} \\[3pt]
	\frac{3}{14} & {-} \frac{27}{49} & \phantom{-} \frac{239}{840} & {-} \frac{57}{70} & \phantom{-} \frac{1369}{588} & {-} \frac{227}{168} & {-} \frac{65}{588} 
\end{pmatrix}m_n^T.$
} \\[10pt]
\begin{align*}
	\begin{pmatrix}
		 \overline{D}^0_n(\onecomp) \\
		 \overline{D}^1_n(\onecomp) \\
		 \overline{D}^2_n(\onecomp) \\
		 \overline{D}^3_n(\onecomp)
	\end{pmatrix} = \frac{3^n}{2^3 281} \begin{pmatrix}
		\frac{10957}{3\cdot 2^2} \\[3pt] 
		\frac{22737599}{3^2 \cdot 13 \cdot 223} \\[3pt] 
		\frac{151 \cdot 220357}{3^2 \cdot 13 \cdot 223} \\[3pt] 
		\frac{5 \cdot 7 \cdot 19 \cdot 5443}{2^2 \cdot 13 \cdot 223}
	\end{pmatrix} - \frac{1}{2\cdot 11^2 \cdot 17} \begin{pmatrix}
		\frac{3^2 \cdot 1063}{2^2} \\[3pt]
		\frac{2120933}{3^2 \cdot 73} \\[3pt]
		\frac{1039\cdot 8111}{3^2\cdot 7\cdot 73} \\[3pt]
		\frac{5\cdot 53 \cdot 7699}{3\cdot 7 \cdot 73}
	\end{pmatrix}+ \frac{5^5\cdot 107 \cdot 375^{-n}}{2\cdot 11^2 \cdot 17 \cdot 281 } \begin{pmatrix}
		\frac{1}{2^2} \\[3pt]
		\frac{5}{3^2} \\[3pt]
		\frac{13}{2^2 \cdot 3^2} \\[3pt]
		\frac{1}{2\cdot 3^2}
	\end{pmatrix}
\end{align*} 
\begin{align*}
	\begin{pmatrix}
		\overline{D}^0_n(\onecomp) \\
		\overline{D}^0_n(\twocompleft) \\
		\overline{D}^0_n(\twocompup) \\
		\overline{D}^0_n(\twocompright) \\
		\overline{D}^0_n(\threecomp)
	\end{pmatrix} = &\begin{pmatrix}
		\frac{10957}{26976} & - \frac{9567}{16456} & \phantom{-} 0 & \phantom{-} \frac{2875}{11616} & \phantom{-} 0 & - \frac{334375}{4624136} \\[3pt]
		\frac{10957}{26976} & \phantom{-} \frac{9567}{16456} & - \frac{363}{392} & - \frac{5405}{11616} & \phantom{-} \frac{55}{196} & \phantom{-} \frac{13954125}{113291332} \\[3pt]
		\frac{10957}{26976} & \phantom{-} \frac{9567}{16456} & - \frac{363}{392} & - \frac{5405}{11616} & \phantom{-} \frac{55}{196} & \phantom{-} \frac{13954125}{113291332} \\[3pt]
		\frac{10957}{26976} & \phantom{-} \frac{9567}{16456} & - \frac{363}{392} & - \frac{5405}{11616} & \phantom{-} \frac{55}{196} & \phantom{-} \frac{13954125}{113291332} \\[3pt]
		\frac{10957}{26976} & \phantom{-} \frac{28701}{16456} & - \frac{363}{196} & \phantom{-} \frac{10603}{11616} & - \frac{99}{98} & - \frac{45504405}{226582664}
	\end{pmatrix}\begin{pmatrix}
		3^n \\
		1 \\
		\left(\frac{3}{5}\right)^n \\
		\left(\frac{3}{25}\right)^n \\
		\left(\frac{1}{25}\right)^n \\
		\left(\frac{1}{375}\right)^n
	\end{pmatrix}, \\
	\begin{pmatrix}
		\overline{D}^1_n(\onecomp) \\
		\overline{D}^1_n(\twocompleft) \\
		\overline{D}^1_n(\twocompup) \\
		\overline{D}^1_n(\twocompright) \\
		\overline{D}^1_n(\threecomp)
	\end{pmatrix} = &\begin{pmatrix}
		\frac{22747599}{58652568} & - \frac{2120933}{5405796} & \phantom{-} 0 & \phantom{-} \frac{2035}{22932} & - \frac{101875}{426888} & \phantom{-} 0 & - \frac{175375}{28716156} & \phantom{-} \frac{1671875}{10404306} \\[3pt]
		\frac{22737599}{58652568} & \phantom{-} \frac{2120933}{5405796} & - \frac{66}{49} & \phantom{-} \frac{1529}{2548} & \phantom{-} \frac{191525}{426888} & - \frac{95}{882} & -\frac{960865}{9572052} & -\frac{23256875}{84968499} \\[3pt]
		\frac{22737599}{58652568} & \phantom{-} \frac{2120933}{1801932} & - \frac{132}{49} & \phantom{-} \frac{6017}{7644} & - \frac{375715}{426888} & \phantom{-} \frac{19}{49} & \phantom{-} \frac{1238333}{3190684} & \phantom{-} \frac{25280225}{56645666}
	\end{pmatrix}\begin{pmatrix}
		3^n \\
		1 \\
		\left(\frac{3}{5}\right)^n \\
		\left(\frac{2}{5}\right)^n \\
		\left(\frac{3}{25}\right)^n \\
		\left(\frac{1}{25}\right)^n \\
		\left(\frac{2}{75}\right)^n \\
		\left(\frac{1}{375}\right)^n
	\end{pmatrix},\\
		\begin{pmatrix}
		\overline{D}^2_n(\onecomp) \\
		\overline{D}^2_n(\twocompleft) \\
		\overline{D}^2_n(\twocompup) \\
		\overline{D}^2_n(\twocompright) \\
		\overline{D}^2_n(\threecomp)
	\end{pmatrix} = &\begin{pmatrix}
		\frac{33273907}{58652568} & - \frac{8427329}{18920286} & \phantom{-} 0 & - \frac{370}{5733} & \phantom{-} \frac{5}{21} & - \frac{43375}{213444} & \phantom{-} 0 & \phantom{-} \frac{175375}{14358078} & - \frac{4346875}{41617224}\\[3pt]
		\frac{33273907}{58652568} & - \frac{18085244}{9460143} & \phantom{-} \frac{675}{392} & - \frac{278}{637} & - \frac{3}{14} & \phantom{-} \frac{81545}{213444} & - \frac{215}{441} & \phantom{-} \frac{960865}{4786026} & \phantom{-} \frac{60467875}{339873996} \\[3pt]
		\frac{33273907}{58652568} & - \frac{3043507}{900966} & \phantom{-} \frac{675}{196} & - \frac{1094}{1911} & \phantom{-} 0 & - \frac{159967}{213444} & \phantom{-} 0 & - \frac{1238333}{1595342} & - \frac{65728585}{226582664}
	\end{pmatrix}\begin{pmatrix}
		3^n \\
		1 \\
		\left(\frac{3}{5}\right)^n \\
		\left(\frac{2}{5}\right)^n \\
		\left(\frac{1}{15}\right)^n \\
		\left(\frac{3}{25}\right)^n \\
		\left(\frac{1}{25}\right)^n \\
		\left(\frac{2}{75}\right)^n \\
		\left(\frac{1}{375}\right)^n
	\end{pmatrix}
\end{align*}
Finally for the expected total number of descendants of Section \ref{sec:expHeight} we obtain
{
\begin{align*}
\begin{pmatrix}
    \overline{D}_n(\onecomp)\\
    \overline{D}_n(\twocompleft)\\
    \overline{D}_n(\twocompup)\\
    \overline{D}_n(\twocompright)\\
    \overline{D}_n(\threecomp)
\end{pmatrix}&= \begin{pmatrix}
	\frac{7259}{3744} & {-} \frac{769}{504} & \phantom{-} 0 & {-} \frac{185}{1638} & {-} \frac{125}{2016} & {-} \frac{5}{21} & \phantom{-} 0\\[3pt] 
	\frac{7259}{3744} & {-} \frac{2579}{504} & \phantom{-} \frac{15}{4} & {-} \frac{139}{182} & \phantom{-} \frac{235}{2016} & \phantom{-} \frac{3}{14} & {-} \frac{5}{36}\\[3pt]
 \frac{7259}{3744} & {-} \frac{2579}{504} & \phantom{-} \frac{15}{4} & {-} \frac{139}{182} & \phantom{-} \frac{235}{2016} & \phantom{-} \frac{3}{14} & {-} \frac{5}{36}\\[3pt]
 \frac{7259}{3744} & {-} \frac{2579}{504} & \phantom{-} \frac{15}{4} & {-} \frac{139}{182} & \phantom{-} \frac{235}{2016} & \phantom{-} \frac{3}{14} & {-} \frac{5}{36}\\[3pt]
    \frac{7259}{3744} & {-} \frac{209}{24} & \phantom{-} \frac{15}{2} & {-} \frac{547}{546} & {-} \frac{461}{2016} & \phantom{-} 0 & \phantom{-} \frac{1}{2}\\[3pt]
\end{pmatrix}\begin{pmatrix}
    3^n \\
    1 \\
    \left(\frac{3}{5}\right)^n \\
    \left(\frac{2}{5}\right)^n \\
    \left(\frac{3}{25}\right)^n \\
    \left(\frac{1}{15}\right)^n \\
    \left(\frac{1}{25}\right)^n
\end{pmatrix}.
\end{align*}

}

\textbf{Funding information.} 
The research of R.~Kaiser and E.~Sava-Huss was funded in part by the Austrian Science Fund (FWF) 10.55776/P34129. For open access purposes, the authors have applied a CC BY public copyright license to any
author-accepted manuscript version arising from this submission.

\textbf{Acknowledgments.} We are very grateful to the two anonymous referees for their suggestions and positive criticism, which substantially improved the quality and the presentation of the paper.

\bibliographystyle{alpha}
\bibliography{lit}

\textsc{Nico Heizmann}, Department of Mathematics, Chemnitz University of Technology, Germany\\
\texttt{nico.heizmann@math.tu-chemnitz.de}

\textsc{Robin Kaiser}, Institut für Mathematik, Universität Innsbruck, Austria.\\
\texttt{Robin.Kaiser@uibk.ac.at}

\textsc{Ecaterina Sava-Huss}, Institut für Mathematik, Universität Innsbruck, Austria.\\
\texttt{Ecaterina.Sava-Huss@uibk.ac.at}
\end{document}